\documentclass[11pt,a4paper]{article}
\date{}
\makeatletter
\@addtoreset{equation}{section}
\makeatother
\usepackage[colorlinks,
            linkcolor=blue,
            anchorcolor=blue,
            citecolor=blue]{hyperref}
\usepackage{graphicx}
\usepackage{hyperref}
\hypersetup{hypertex=true,
            colorlinks=true,
            linkcolor=blue,
            anchorcolor=blue,
            citecolor=blue}
\usepackage{float}
\usepackage{hyperref,amsmath,amssymb,amscd}
\usepackage{amssymb,amsmath,enumerate}
\usepackage{indentfirst}
\usepackage{url}
\usepackage[capitalise]{cleveref}
\usepackage{graphicx}
\usepackage{amsthm}
\usepackage{cite}
\newtheorem{theorem}{Theorem}[section]

\newtheorem{lemma}[theorem]{Lemma}
\newtheorem{proposition}[theorem]{Proposition}
\theoremstyle{definition}

\newtheorem{remark}[theorem]{\bf Remark}
\allowdisplaybreaks[4]

\begin{document}
\title{\bf
{Normalized solutions to a focusing biharmonic Schr\"odinger equation with mixed dispersion and Sobolev critical growth}}
\author{{Jianlun Liu$^1$\footnote{E-mail address: jianlunliumath@163.com (J. Liu)},\ \ Hong-Rui Sun$^1$\footnote{Corresponding author. E-mail address: hrsun@lzu.edu.cn (H.R. Sun)}\ \ and\ \ Ziheng Zhang$^2$\footnote{E-mail address: zhzh@mail.bnu.edu.cn (Z. Zhang)}}\\
{\small \emph{$^1$School of Mathematics and Statistics,\  Lanzhou University,\ Lanzhou {\rm730000}, China}}\\
{\small \emph{$^2$School of Mathematical Sciences,\  Tiangong University,\ Tianjin {\rm300387}, China}}\\
}
\maketitle
\baselineskip 17pt

{\bf Abstract}: This paper investigates normalized solutions to a focusing biharmonic Schr\"odinger equation with mixed dispersion and combined nonlinearities involving a lower-power perturbation and a Sobolev critical term. Under a suitable smallness condition on the prescribed mass and in the range \(2<p<2+\frac{4}{N}\), we establish the existence of a local minimizer of the associated energy functional on \(S_r(c)\) and infinitely many pairs of radial normalized solutions by combining the truncation technique, the concentration-compactness principle and genus theory. In particular, our results extend those of Alves et al. (Anal. Appl., (2025), doi: 10.1142/S0219530525500447) by removing the additional assumptions imposed in their work. In the range \(2+\frac{8}{N}<p<4^*\), we prove the existence of bounded radial constrained Palais-Smale sequences at mountain pass levels below the critical Sobolev threshold for sufficiently large \(\mu\). We further show that their nonzero weak limits are radial weak solutions, with strong convergence in the derivative norms and in the critical Lebesgue norm. Moreover, under an additional condition \((N)\) ensuring the strict negativity of the limiting Lagrange multiplier, the weak limit preserves the prescribed \(L^2\)-mass, and the Palais-Smale sequence converges strongly in \(H^2(\mathbb R^N)\). These results extend the previous studies of Chang et al. (arXiv preprint, (2023), arXiv: 2305.00327v1) and Bonheure et al. (Trans. Amer. Math. Soc., (2019), doi:10.1090/tran/7769).

\textbf{Key words and phrases.} Biharmonic Schr\"{o}dinger equation; Mixed dispersion; Normalized solutions; Sobolev critical growth.

\textbf{2020 Mathematics Subject Classification}. 35A15, 35J30, 35J35, 35J60.\\

\section{Introduction and main results}
In this paper, we investigate the following focusing biharmonic Schr\"{o}dinger equation with mixed dispersion and Sobolev critical growth:
\begin{equation}\label{eqn:BS-equation-L2-Super+Critical}
\begin{cases}
	{\Delta}^2u-\Delta u-\lambda u=\mu|u|^{p-2}u+|u|^{4^*-2}u\ \ \mbox{in}\ \mathbb{R}^N, \\[0.1cm]
	\int_{\mathbb{R}^N} u^2 dx = c,
\end{cases}
\end{equation}
where $N \geq 5$, $\mu,c>0$, $2<p<4^*:=\frac{2N}{N-4}$ and $\lambda \in \mathbb{R}$ is a Lagrange multiplier. The interest in studying problem (\ref{eqn:BS-equation-L2-Super+Critical}) comes from seeking standing waves with the form $\psi(t,x)=e^{-i\lambda t}u(x)$ to the following time-dependent biharmonic nonlinear Schr\"{o}dinger (BNLS) equation:
\begin{equation}\label{eqn:intial-equation}
i\partial_{t}\psi-\gamma\Delta^2 \psi+\beta\Delta\psi+\kappa f(|\psi|)\psi=0,
\end{equation}
where $\gamma>0$, $\beta,\kappa\in\mathbb R$, $f(s)=\mu|s|^{p-2}+|s|^{4^*-2}$, $\mu>0$ and $i$ denotes the imaginary unit. For $\gamma=\beta=\kappa=1$, the
standing-wave ansatz $\psi(t,x)=e^{-i\lambda t}u(x)$ yields (\ref{eqn:BS-equation-L2-Super+Critical}). This equation was proposed in \cite{Karpman1996,Karpman-Shagalov2000}, where the role of a small fourth-order dispersion term during the propagation of intense laser beams in a bulk medium with Kerr nonlinearity was considered. The biharmonic operator $\Delta^2$ was also used to reveal the effects of higher-order dispersion terms in the mixed-dispersion fourth-order Schr\"{o}dinger equations. Actually, the case $\beta=0$ in (\ref{eqn:intial-equation}) was considered earlier, in \cite{Ivanov-Kosevich1983,Turitsyn1985}, in the context of the stability of solitons in magnetic materials when the effective quasi-particle mass becomes infinite. The case $\beta<0$ arises in the approximation to the vectorial nonlinear Helmholtz equation as a nonparaxial correction to the second-order nonlinear Schr\"{o}dinger equation, while the case $\beta>0$ relates to faster transmission in optical fiber arrays, see \cite{Fibich-Ilan-Papaniclaou2002,Miao-Xu-Zhao2009,Pausader2009,Zhang-Zheng2010} and the references listed therein for more physical backgrounds. Moreover, from the physical standpoint, the real number $\kappa$ refers to the defocusing versus focusing regime, that is, $\kappa=-1$ corresponds to the defocusing problem, while $\kappa=1$ stands for the focusing problem.

In the study of nonlinear dispersive equations, significant attention has been directed toward standing wave solutions with prescribed mass constraint for (\ref{eqn:intial-equation}), where the mass parameter often carries physical interpretations--for instance, representing the total number of atoms in Bose-Einstein condensation or the power supply in nonlinear optics. Such solutions, termed normalized solutions in the literature, serve as critical tools for analyzing dynamical characteristics of stationary solutions, particularly orbital stability/instability. The exponent \(\bar p:=2+\frac{8}{N}\) is the \(L^2\)-critical threshold for the pure-power biharmonic Schr\"odinger equation. For problem (\ref{eqn:BS-equation-L2-Super+Critical}), however, the additional Sobolev-critical term makes the energy functional \(I\) unbounded from below on \(S(c)\) for every \(2<p<4^*\), including the \(L^2\)-subcritical range \(2<p<\bar p\). Consequently, global minimization on \(S(c)\) is not applicable. Nevertheless, in the lower-power regime \(2<p<2+\frac{4}{N}\), we seek local constrained minimizers in a suitably chosen bounded subset of \(S(c)\) under an appropriate smallness condition on the prescribed mass, as developed in Section \ref{sec:proof-main-theorem-1}. Here, we refer the readers to \cite{Bellazzini-Visciglia2010,Bonheure-Casteras-Gou-Jeanjean2019,Bonheure-Casteras-Santos-Nascimento2018,Boussaid-Fernandez-Jeanjean2019,
Chang-Hajaiej-Ma-Song2023,Fernandez-Jeanjean-Mandel-Maris2022,Liu-Zhang2023,Luo-Yang2023,
Luo-Zhang2022,Luo-Zheng-Zhu2023,Ma-Chang2022,Ma-Chang-Feng2024,Phan2018,Zhang-Wang2024,Zhu-Zhang-Yang2010} to learn about the latest developments of the existence of normalized solutions for the type of problem (\ref{eqn:BS-equation-L2-Super+Critical}).

As pointed out in \cite{Luo-Zhang2022}, the presence of the second-order dispersion term introduces an additional difficulty in the variational analysis. For real-valued profiles, the nonlinearity in the stationary equation can be expressed as
\[
g(s):=f(|s|)s=\mu|s|^{p-2}s+|s|^{4^*-2}s,
\]
with its primitive given by
\[
G(s):=\int_0^s g(t)dt=\frac{\mu}{p}|s|^p+\frac{1}{4^*}|s|^{4^*}.
\]
When \(\beta\neq 0\), the associated energy involves both \(\|\Delta u\|_2^2\) and \(\|\nabla u\|_2^2\). In contrast, the gradient term disappears when \(\beta=0\), leaving only the fourth-order dispersion contribution. This distinction explains why considerable attention has been devoted to the case \(\gamma=1\) and \(\beta=0\) in (\ref{eqn:intial-equation}). In this setting, Bellazzini and Visciglia \cite{Bellazzini-Visciglia2010} investigated the following minimization problem
\begin{equation}\label{eqn:Minimization-problem-VQ}
\inf_{u\in S(c)}\Bigl(\frac{1}{2}\int_{\mathbb{R}^N}|\Delta u|^2dx+\frac{1}{2}\int_{\mathbb{R}^N}V(x)|u|^2dx-\frac{1}{p}\int_{\mathbb{R}^N}Q(x)|u|^{p}dx\Bigr)
\end{equation}
where
\begin{equation}\label{eqn:Defn-Sc}
S(c):=\Big\{u \in H^2(\mathbb{R}^N):\int_{\mathbb R^N}u^2dx=c\Big\}.
\end{equation}
A minimizer of (\ref{eqn:Minimization-problem-VQ}) is a ground state normalized solution of the following problem
\begin{equation}\label{eqn:BS-F-equation}
\begin{cases}
	{\Delta}^2u-\lambda u+V(x)u=Q(x)|u|^{p-2}u \ \ \mbox{in}\ \mathbb{R}^N, \\[0.1cm]
	\int_{\mathbb{R}^N} u^2 dx = c.
\end{cases}
\end{equation}
Under the assumptions \(2<p<\bar p\), \(V,Q\in L^\infty(\mathbb R^N)\), and \(Q(x)\geq0\) almost everywhere in \(\mathbb R^N\), they established the existence and compactness of the set of minimizers for sufficiently large prescribed mass. Their result also requires the existence of \(\lambda_0>0\) such that \(0<\operatorname{meas}\{Q(x)>\lambda_0\}<\infty\). In the autonomous case \(V\equiv0\) and \(Q\equiv1\), they obtained the existence of a minimizer for every \(c>0\). However, this case must be distinguished from the preceding nonautonomous setting. Indeed, translation invariance prevents the set of minimizers from being compact in \(H^2(\mathbb R^N)\) without accounting for translations. Later Phan \cite{Phan2018} investigated the following problem with $L^2$-critical nonlinearity
\begin{equation}\label{bSV-equation}
\begin{cases}
	{\Delta}^2u-\lambda u+V(x)u=a |u|^{\frac{8}{N}}u\ \ \mbox{in}\ \mathbb{R}^N, \\[0.1cm]
	\int_{\mathbb{R}^N} u^2 dx = 1,
\end{cases}
\end{equation}
where the parameter $a>0$ stands for the strength of the attraction of the system and $V(x)$ stands for an external potential satisfying either
\begin{itemize}
\item[$(V1)$] $V\geq 0$, $V(x)\rightarrow \infty$ as $|x|\rightarrow \infty$ or
\item[$(V2)$] ${\rm ess} \inf V\leq 0$, $\min\{V,0\}\in L^{p_1}(\mathbb{R}^N)+L^{p_2}(\mathbb{R}^N)$ and $\max\Bigl\{1,\frac{N}{4}\Bigr\}<p_1<p_2<\infty$.
\end{itemize}
Under the conditions presented above, the author revealed that there exists a constant $a_* \in (0, a^*)$ such that for all
$a_* < a < a^*<\infty$, problem (\ref{bSV-equation}) admits at least one ground state normalized solution, where $a^*$ is the optimal constant in the
Gagliardo-Nirenberg interpolation inequality. Recently, Ma and Chang \cite{Ma-Chang2022} investigated the existence of normalized ground states for the biharmonic Schr\"{o}dinger equation with combined \(L^2\)-subcritical and Sobolev critical nonlinearities:
\begin{equation}\label{eqn:BS-critical+L2-equation}
\begin{cases}
	{\Delta}^2u-\lambda u = \mu |u|^{p-2}u+|u|^{4^{*}-2}u\ \ \mbox{in}\ \mathbb{R}^N, \\[0.1cm]
	\int_{\mathbb{R}^N} u^2 dx = c,  \\[0.1cm]
\end{cases}
\end{equation}
where $N\geq 5$, $\mu,c>0$, $2<p<\bar{p}$. Actually, the authors showed that, for any $\mu>0$, there exist $c_0=c_0(\mu,p,N) > 0$ and $\rho_0=\rho(c_0)>0$ such that problem (\ref{eqn:BS-critical+L2-equation}) has a ground state normalized solution $u$ on $S(c)$ satisfying $\|\Delta u\|_2^2\leq \rho_0$, when $c$ belongs to $(0, c_0)$. Building on the results of Ma and Chang \cite{Ma-Chang2022}, Liu and Zhang \cite{Liu-Zhang2023} extended the analysis to the \(L^2\)-supercritical perturbation. By estimating the mountain-pass level and establishing the Palais–Smale condition, they proved the existence of normalized solutions for sufficiently large \(\mu\). Furthermore, they investigated the asymptotic behavior of the corresponding mountain-pass energy as \(\mu\to\infty\) and \(c\to+\infty\). At the same time, after constructing an appropriate test function and making some energy estimates, Chang et al. \cite{Chang-Hajaiej-Ma-Song2023} proved the existence of normalized ground state solutions for problem (\ref{eqn:BS-critical+L2-equation}) when $\bar{p}\leq p<4^*$ and
showed that all ground states correspond to the local minima of the associated energy functional restricted to the Poho\v{z}aev set for any $\mu>0$. Moreover, they also revealed that the standing waves are strongly unstable by blowup for cases $\bar{p}\leq p<4^*$. For a general $L^2$-supercritical nonlinearity \(g\), Zhang and Wang \cite{Zhang-Wang2024} considered the following problem:
\begin{equation}\label{eqn:BS-critical+L2-equation-general}
\begin{cases}
	{\Delta}^2u-\lambda u = g(u)\ \ \mbox{in}\ \mathbb{R}^N, \\[0.1cm]
	\int_{\mathbb{R}^N} u^2 dx = c,  \\[0.1cm]
\end{cases}
\end{equation}
where $N\geq 5$, $c>0$ and the nonlinearity $g$ satisfies:
\begin{itemize}
\item[$(G1)$] $g: \mathbb{R} \rightarrow \mathbb{R}$ is continuous and odd;
\item[$(G2)$] there exist $\bar{p}<\alpha \leq \beta<4^*:=\frac{2 N}{N-4}$
such that for any $s \neq 0$, $0<\alpha G(s) \leq g(s) s \leq \beta G(s)$,
where $G(s):=\int_0^s g(t) dt$;
\item[$(G3)$] the function defined by $\widetilde{G}(s):=\frac{1}{2} g(s) s-G(s)$ is of class $C^1$ and
$\widetilde{G}^{\prime}(s) s \geq \alpha \widetilde{G}(s)$ for any $s \in \mathbb{R}$, where $\alpha$ is given in $(G2)$.
\end{itemize}
After showing that the new constraint is natural and verifying the compactness of the minimizing sequence, they obtained the existence of normalized ground state solutions.

Biharmonic Schr\"{o}dinger equations with mixed dispersion have also attracted considerable attention, particularly when \(\beta\neq0\) in (\ref{eqn:intial-equation}). In this setting, Bonheure et al. \cite{Bonheure-Casteras-Santos-Nascimento2018} investigated the existence of normalized ground states as constrained energy minimizers for the following problem:
\begin{equation}\label{eqn:parameter-equation}
\begin{cases}
	\gamma\Delta^2 u-\beta \Delta u- \lambda u=|u|^{p-2}u\ \ \mbox{in}\ \mathbb{R}^N, \\[0.1cm]
	\int_{\mathbb{R}^N} u^2 dx = c.
\end{cases}	
\end{equation}
For \(\gamma,\beta>0\) and \(2<p<\bar p\), they established the existence of normalized ground states under suitable conditions on the prescribed mass. Subsequently, Bonheure et al. \cite{Bonheure-Casteras-Gou-Jeanjean2019} investigated the mass-critical and mass-supercritical regimes \(\bar p\leq p<4^*\) for \(\gamma>0\) and \(\beta=1\). By applying a minimax principle based on a homotopy-stable family, they obtained existence and multiplicity results for normalized solutions under the assumptions of their theorems. Assuming that $\gamma=c=1, \beta \in \mathbb{R}$ and $2<p \leq \bar{p}$, Luo et al.\cite{Luo-Zheng-Zhu2023} studied the minimization problem
\begin{equation}\label{eqn:m-problem}
m(c,\gamma,\beta):=\inf _{u \in S(c)} E_{\gamma,\beta}(u)
\end{equation}
by using the profile decomposition of bounded sequences in $H^2(\mathbb R^N)$ established in \cite{Zhu-Zhang-Yang2010}, where
$$
E_{\gamma,\beta}(u):=\frac{\gamma}{2}\|\Delta u\|^2_2+\frac{\beta}{2}\|\nabla u\|^2_2-\frac{1}{p}\|u\|^p_p.
$$
They revealed that $m(c,\gamma,\beta)$ is achieved in five cases:
\begin{itemize}
\item[$(i)$] $c=1$, $2<p<2+\frac{4}{N}$ and $\beta \in(0,\infty)$;
\item[$(ii)$] $c=1$, $2+\frac{4}{N} \leq p<\bar{p}$ and $\beta \in(0,\beta_0)$ for some $\beta_0>0$;
\item[$(iii)$] $c=1,2<p<\bar{p}$ and $\beta=0$;
\item[$(iv)$] $c=1$, $2<p<\bar{p}$ and $\beta \in[-\hat{\beta}_0,0)$ for some $\hat{\beta}_0>0$;
\item[$(v)$] $0<c<c_N$, $p=\bar{p}$ and $\beta \in[-\tilde{\beta}_0,0)$ for some $c_N$, $\tilde{\beta}_0>0$.
\end{itemize}
Boussa\"{\i}d et al. \cite{Boussaid-Fernandez-Jeanjean2019} subsequently revisited problem (\ref{eqn:m-problem}) for \(\gamma>0\), \(\beta<0\), and \(2<p<\bar p\). Their analysis improved upon the results of Luo et al.  \cite{Luo-Zheng-Zhu2023} by removing certain restrictions on the prescribed mass and the second-order dispersion coefficient. In particular, by ruling out the vanishing of minimizing sequences, they established the existence of minimizers for every \(c>0\) in the \(L^2\)-subcritical regime. Luo and Yang \cite{Luo-Yang2023} investigated problem (\ref{eqn:parameter-equation}) in the \(L^2\)-supercritical regime \(\bar p<p<\min\{4,4^*\}\) with \(\gamma=1\) and \(\beta<0\). Under a suitable smallness condition involving the prescribed mass and the second-order dispersion coefficient, they established the existence of two radial normalized solutions, including a ground state and an excited state. Here, the radial Sobolev space is defined by
\begin{equation}\label{eqn:H2r-define}
H^2_{r}(\mathbb{R}^N):=\{u\in H^2(\mathbb{R}^N):u\ \mbox{is radially symmetric with respect to 0}\}.
\end{equation}
Fern\'{a}ndez et al. \cite{Fernandez-Jeanjean-Mandel-Maris2022} established non-homogeneous Gagliardo–Nirenberg inequalities and applied them to the study of normalized standing waves for mixed-dispersion biharmonic Schr\"odinger equations. Their analysis provided optimal existence results for global energy minimizers in the mass-subcritical and \(L^2\)-critical regimes. Furthermore, they investigated the existence of local minimizers in the \(L^2\)-supercritical regime when the second-order and fourth-order dispersion terms have the same sign in the equation. For (\ref{eqn:parameter-equation}), this corresponds to \(\gamma>0\) and \(\beta<0\). The main results discussed above are summarized in the following Table \ref{1,2}.

\begin{table}[H]\label{1,2}
	\centering
	\caption{Existing results for (\ref{eqn:parameter-equation})}
\resizebox{\linewidth}{!}{%
	\begin{tabular}{|c|c|c|c|c|}
		\hline
		$\gamma$ & $\beta$ & $p$ & Type of solutions & Reference \\ \hline
		
		$\gamma>0$ & $\beta>0$ & $2<p<\bar{p}$ 
		& A ground state 
		& Bonheure et al. \cite{Bonheure-Casteras-Santos-Nascimento2018} \\ \hline
		
		$\gamma>0$ & $\beta=1$ & $\bar{p}\leq p<4^*$ 
		& \begin{tabular}[c]{@{}c@{}}A ground state and \\ infinitely many radial solutions\end{tabular}
		& Bonheure et al. \cite{Bonheure-Casteras-Gou-Jeanjean2019} \\ \hline
		
		$\gamma=1$ 
		& \begin{tabular}[c]{@{}c@{}}
			$\beta>0$\\
			$0<\beta<\beta_0$\\
			$\beta=0$\\
			$-\hat{\beta}_0\leq\beta<0$\\
			$-\tilde{\beta}_0\leq\beta<0$
		\end{tabular}
		& \begin{tabular}[c]{@{}c@{}}
			$2<p<2+\frac{4}{N}$\\
			$2+\frac{4}{N}\leq p<\bar{p}$\\
			$2<p<\bar{p}$\\
			$2<p<\bar{p}$\\
			$p=\bar{p}$
		\end{tabular}
		& A ground state
		& Luo et al. \cite{Luo-Zheng-Zhu2023} \\ \hline
		
		$\gamma>0$ & $\beta<0$ & $2<p<\bar{p}$ 
		& A ground state 
		& Boussa\"{\i}d et al. \cite{Boussaid-Fernandez-Jeanjean2019} \\ \hline
		
		$\gamma=1$ & $\beta<0$ 
		& $\bar{p}<p<\min\{4,4^*\}$ 
		& \begin{tabular}[c]{@{}c@{}}A local minimizer and a \\ mountain-pass type solution\end{tabular}
		& Luo and Yang \cite{Luo-Yang2023} \\ \hline
		
		$\gamma=1$ & $\beta=-2$ 
		& \begin{tabular}[c]{@{}c@{}}
			$2<p\leq\bar{p}$\\
			$\bar{p}<p<\min\{4,4^*\}$
		\end{tabular}
		& \begin{tabular}[c]{@{}c@{}}
			A ground state\\
			A local minimizer
		\end{tabular}
		& Fern\'{a}ndez et al. \cite{Fernandez-Jeanjean-Mandel-Maris2022} \\ \hline
	\end{tabular}%
}
\end{table}

For more general nonlinearity, we notice that Luo and Zhang \cite{Luo-Zhang2022} investigated the following problem:
\begin{equation}\label{eqn:BS-beta-gamma-general-Luo-Zhang}
\left\{\begin{array}{ll}
\gamma\Delta^2 u-\beta \Delta u-\lambda u=g(u) \,\,\,\, \mbox{in}\,\,\, \mathbb{R}^{N}, \\[0.2cm]
\int_{\mathbb{R}^N}u^2 dx=c,
\end{array}\right.
\end{equation}
where $\gamma>0$ and $\beta\in \mathbb{R}$. Under suitable \(L^2\)-subcritical growth assumptions and additional structural conditions on \(g\), they established existence, nonexistence, and stability results for normalized solutions. Their analysis also covered certain combined-power nonlinearities in the \(L^2\)-subcritical regime. However, the focusing combination of a lower-power term and a Sobolev critical term falls outside the framework considered there. In this direction, Ma et al. \cite{Ma-Chang-Feng2024} investigated the following mixed dispersion problem with combined nonlinearities:
\begin{equation}\label{eqn:BS-beta-gamma-general-Ma-Chang-Feng}
\left\{\begin{array}{ll}
\gamma\Delta^2 u-\beta \Delta u-\lambda u=\mu|u|^{p-2}u+|u|^{4^*-2}u \,\,\,\, \mbox{in}\,\,\, \mathbb{R}^{N}, \\[0.2cm]
\int_{\mathbb{R}^N}u^2 dx=c,
\end{array}\right.
\end{equation}
where $N\geq 5$, $\gamma,\beta,\mu>0$ and $2<p<2+\frac{4}{N}$. They established the existence of a ground state that corresponds to a local minimum of the associated functional by analyzing the subadditivity of the ground state energy with respect to the prescribed mass and employing a constrained minimization method. However, the multiplicity of normalized solutions in this regime was not addressed in \cite{Ma-Chang-Feng2024}. Motivated by this observation, we investigate the existence of multiple normalized solutions for \(2<p<2+\frac4N\). We also examine the existence of normalized solutions in the \(L^2\)-supercritical regime. To formulate our main results, we set \(\gamma=\beta=1\) in (\ref{eqn:BS-beta-gamma-general-Ma-Chang-Feng}), which reduces the equation to problem (\ref{eqn:BS-equation-L2-Super+Critical}). The normalized weak solutions of (\ref{eqn:BS-equation-L2-Super+Critical}) can then be characterized as constrained critical points of the associated energy functional. More precisely, a critical point of \(I\) on \(S(c)\) satisfies (\ref{eqn:BS-equation-L2-Super+Critical}) for a suitable Lagrange multiplier \(\lambda\in\mathbb R\). The functional \(I\in \mathcal C^1(H^2(\mathbb R^N),\mathbb R)\) is defined by
\begin{equation}\label{eqn:Defn-Iu}
I(u):=\frac{1}{2}\int_{\mathbb{R}^N}|\Delta u|^2 dx+\frac{1}{2}\int_{\mathbb{R}^N}|\nabla u|^2 dx-\frac{\mu}{p}\int_{\mathbb{R}^N}|u|^p dx-\frac{1}{4^*}\int_{\mathbb{R}^N}|u|^{4^*}dx,\ \  u\in H^2({\mathbb{R}^N})
\end{equation}
restricted to the constraint
$$
S(c)=\Bigl\{u\in H^2(\mathbb{R}^N):\int_{\mathbb{R}^N}u^2dx=c\Bigr\}.
$$
At this moment, we can state our main results as follows.

\begin{theorem}\label{Thm:normalized-Subcritical-solutions-existence}
Let \(\mu>0\), \(N\geq5\), and \(2<p<2+\frac{4}{N}\). For every \(0<c<c^*\), there exists \(u\in V_r(c)\) such that
\[
I(u)=\inf_{v\in V_r(c)}I(v)<0.
\]
Moreover, \(u\) is a local minimizer of \(I\) on \(S_r(c)\) and a constrained critical point of \(I\) on \(S(c)\). Consequently, \(u\) is a normalized weak solution of (\ref{eqn:BS-equation-L2-Super+Critical}), and the corresponding Lagrange multiplier satisfies \(\lambda<0\). Here, \(I\), \(c^*\), and \(V_r(c)\) are defined in (\ref{eqn:Defn-Iu}), (\ref{eqn:Defn-c*}) and (\ref{eqn:Vr0-define}), respectively.
\end{theorem}

\begin{theorem}\label{Thm:normalized-Subcritical-solutions-infinity}
Assume that \(N\geq5\), \(2<p<2+\frac4N\), \(\mu>0\), and \(0<c<c^*\). Then, for every \(m\in\mathbb N^+\), \(I|_{V_r(c)}\) admits at least \(m\) distinct pairs of constrained critical points \(\{\pm u_j\}_{j=1}^m\). Moreover, each \(u_j\) satisfies
\[
I(u_j)<0,\qquad \lambda_j<0,
\]
where \(\lambda_j\) is the corresponding Lagrange multiplier.
\end{theorem}
\begin{remark}
{\rm
As stated above, Ma et al. \cite[Theorem 1.2]{Ma-Chang-Feng2024} discussed the existence of ground state normalized solutions for problem (\ref{eqn:BS-equation-L2-Super+Critical}). Inspired by \cite{Alves-Ji-Miyagaki2021}, we use a different approach from those in \cite{Ma-Chang-Feng2024} to obtain the existence of normalized solution to problem (\ref{eqn:BS-equation-L2-Super+Critical}) and also give the sign of its corresponding Lagrange multiplier with a novel method in Theorem \ref{Thm:normalized-Subcritical-solutions-existence}. In contrast to \cite{Alves-Ji-Miyagaki2021}, we improve the methods of Alves et al. and remove the further constraints given in (3.2) therein, giving a more general result in Theorems \ref{Thm:normalized-Subcritical-solutions-existence} and \ref{Thm:normalized-Subcritical-solutions-infinity}.
}
\end{remark}
Ma and Chang \cite[Theorem 1.2]{Ma-Chang2022} established the existence of normalized ground states for problem (\ref{eqn:BS-critical+L2-equation}) in the \(L^2\)-subcritical regime. Their result motivates a further investigation of the multiplicity of normalized solutions. Although the second-order dispersion term is absent from (\ref{eqn:BS-critical+L2-equation}), the corresponding energy functional still contains a Sobolev critical nonlinearity. Consequently, a suitable local variational framework and compactness analysis are needed. By adapting the variational arguments used in the proof of Theorem \ref{Thm:normalized-Subcritical-solutions-infinity} to the case without the second-order dispersion term, we obtain the following result. The main changes concern the choice of the local variational region and the corresponding scaling estimates. For brevity, we omit the proof, since it follows the same lines as that of Theorem \ref{Thm:normalized-Subcritical-solutions-infinity}, with the modifications described above.
\begin{theorem}\label{Thm:normalized-Subcritical-solutions-infinity-beta=0}
Assume that $N\geq 5$ and $2<p<\bar{p}$. For given $k\in\mathbb N^+$, there exists $\mu_0>0$ such that $I_0|_{\Lambda_\rho(c)}$ admits at least $k$ critical points $u_j$ $(j=1,2,\ldots,k)$ such that $I_0(u_j)<0$ when $0<c<c_0$, where $\Lambda_\rho(c)$ and $c_0$ are given in \cite[Theorem 1.2]{Ma-Chang2022} and
$$
I_0(u):=\frac{1}{2}\int_{\mathbb{R}^N}|\Delta u|^2 dx-\frac{\mu}{p}\int_{\mathbb{R}^N}|u|^p dx-\frac{1}{4^*}\int_{\mathbb{R}^N}|u|^{4^*}dx.
$$
Moreover, the corresponding Lagrange multiplier $\lambda_j<0$. 
\end{theorem}
We now turn to the \(L^2\)-supercritical regime \(\bar p<p<4^*\). As discussed above, the energy functional \(I\) is unbounded from below on \(S(c)\). To establish a suitable variational framework in this regime, we introduce the following Poho\v{z}aev manifold
$$
\mathcal P(c):=\{u\in S(c):P(u)=0\},
$$
where
\begin{equation}\label{eqn:Pohozaev-identify}
P(u):=\int_{\mathbb{R}^N}|\Delta u|^2 dx+\frac{1}{2}\int_{\mathbb{R}^N}|\nabla u|^2 dx-\mu\gamma_p\int_{\mathbb{R}^N}|u|^p dx-\int_{\mathbb{R}^N}|u|^{4^*} dx.
\end{equation}
Every constrained critical point of \(I\) on \(S(c)\) satisfies the Poho\v{z}aev identity \(P(u)=0\) (see \cite[Lemma 2.1]{Bonheure-Casteras-Gou-Jeanjean2017}). We therefore use the radial Poho\v{z}aev set \(\mathcal P_r(c)\) to characterize the constrained mountain pass level. In the compactness analysis, strong convergence of the derivative terms does not by itself guarantee preservation of the prescribed \(L^2\)-mass. To address this issue, we introduce an additional condition \((N)\) on the Palais-Smale sequence.

\((N)\) Let \(\{u_n\}\subset S_r(c)\) be a bounded sequence satisfying \(I(u_n)\to\gamma_\mu(c)\), \(P(u_n)\to0\) and \(\left\|(I|_{S_r(c)})'(u_n)\right\|_{(T_{u_n}S_r(c))'}\to0\), where \(\gamma_\mu(c)\) is the mountain pass level introduced in Section \ref{sec:proof-main-theorem-3-and-4}. Furthermore, assume that for the fixed parameter  \(\mu\),
\[
\liminf_{n\to\infty}\Big[\mu(1-\gamma_p)\|u_n\|_p^p-\frac12\|\nabla u_n\|_2^2\Big]>0.
\]

\begin{theorem}\label{Thm:normalized-bScritical-solutions-L2super}
Let \(N\geq5\), \(c>0\) and \(\bar p<p<4^*\). There exists \(\mu_0=\mu_0(N,p,c)>0\) such that for every \(\mu>\mu_0\),
\[
0<m_r(c)=\gamma_\mu(c)<\frac{2}{N}S^\frac{N}{4}.
\]
Moreover, there exists a bounded sequence \(\{u_n\}\subset S_r(c)\) satisfying \(I(u_n)\to m_r(c)\), \(P(u_n)\to0\) and \(\left\|(I|_{S_r(c)})'(u_n)\right\|_{(T_{u_n}S_r(c))'}\to0\). After passing to a subsequence, there exist \(u_{\mu,c}\in H_r^2(\mathbb R^N)\setminus\{0\}\) and \(\lambda_{\mu,c}\in\mathbb R\) such that
\[
\Delta^2u_{\mu,c}-\Delta u_{\mu,c}-\lambda_{\mu,c}u_{\mu,c}-\mu|u_{\mu,c}|^{p-2}u_{\mu,c}-|u_{\mu,c}|^{4^*-2}u_{\mu,c}=0\quad \mbox{in}\quad \mathbb R^N.
\]
Furthermore,
\[
P(u_{\mu,c})=0,\qquad I(u_{\mu,c})=m_r(c),
\]
and
\[
0<\|u_{\mu,c}\|_2^2\leq c,\qquad \lambda_{\mu,c}\big(c-\|u_{\mu,c}\|_2^2\big)=0.
\]
If the Palais-Smale sequence above can be chosen to satisfy condition (N), then
\[
\lambda_{\mu,c}<0,\qquad
\|u_{\mu,c}\|_2^2=c,
\]
and
\[
u_n\to u_{\mu,c}\quad\text{ in}\quad H^2(\mathbb R^N).
\]
Consequently, \(u_{\mu,c}\) is a radial normalized weak solution of  (\ref{eqn:BS-equation-L2-Super+Critical}) satisfying
\[
0<I(u_{\mu,c})<\frac{2}{N}S^\frac{N}{4}.
\]
\end{theorem}

\begin{remark}
{\rm
The analysis of the \(L^2\)-supercritical regime is based on a constrained mountain pass construction. The mass preserving scaling yields a bounded Palais-Smale sequence satisfying the Poho\v zaev condition asymptotically. A fixed radial test function further shows that the mountain pass level \(m_r(c)\) tends to zero as \(\mu\to+\infty\). Consequently, this level lies below the critical Sobolev threshold for sufficiently large \(\mu\). Although this estimate excludes a nonzero critical Sobolev defect in the compactness argument, it does not by itself guarantee preservation of the prescribed \(L^2\) mass. The latter is obtained under the additional condition \((N)\) in Theorem \ref{Thm:normalized-bScritical-solutions-L2super}. 
}
\end{remark}

The remainder of this paper is organized as follows. Section \ref{sec:preliminary} introduces the notation and preliminary results used throughout the analysis. In Section \ref{sec:proof-main-theorem-1}, we establish the existence of a local constrained minimizer by combining a bounded variational region with compactness at negative energy levels. Section \ref{sec:proof-main-theorem-2} develops the genus-based minimax framework for investigating multiplicity. Finally, Section \ref{sec:proof-main-theorem-3-and-4} constructs a bounded radial Palais-Smale sequence, estimates its mountain pass level, and proves strong convergence of the derivative terms and in the critical Lebesgue norm. Strong convergence in \(H^2(\mathbb R^N)\) and the normalized solution conclusion are obtained under the additional condition \((N)\) in Theorem \ref{Thm:normalized-bScritical-solutions-L2super}.

\section{Preliminaries}\label{sec:preliminary}
Denote by ${L}^s(\mathbb{R}^N)$, $s \in [1,\infty)$, the usual Lebesgue space with its norm
$$
\|u\|_s := \left(\int_{\mathbb{R}^N}|u|^sdx\right)^\frac{1}{s}.
$$
In the variational analysis, all functions are real-valued.
Define the function space
$$
H^2(\mathbb{R}^N) :=\Big\{ u\in L^2(\mathbb{R}^N):\nabla u \in {L}^2(\mathbb{R}^N),\ \Delta u \in {L}^2(\mathbb{R}^N)\Big\}
$$
with the usual norm
$$
\|u\|_{H^2} := \Big(\int_{\mathbb{R}^N}|\Delta u|^2+|\nabla u|^2 +u^2 dx\Big)^\frac{1}{2}.
$$
For $N \geq 5$ and $u \in H^2(\mathbb{R}^N)$, one has the following Gagliardo-Nirenberg inequality (\!\!\cite[Theorem in Lecture II]{Nirenberg1959})
\begin{equation}\label{eqn:GNinequality}
\|u\|_t \leq C_{N,t}\|\Delta u\|^{\gamma_t}_{2}\|u\|^{1-\gamma_t}_2,
\end{equation}
where $t\in(2,4^*]$ and $\gamma_t:=\frac{N}{2}(\frac{1}{2}-\frac{1}{t})$. Particularly, when $t = 4^*$, we find that $\gamma_{4^*}=1$ and
\begin{equation}\label{eqn:Defn-S}
(C_{N,4^*})^{-2}=S:=\inf_{u\in H^2(\mathbb R^N)\setminus\{0\}}\frac{\int_{\mathbb R^N}|\Delta u|^2dx}{(\int_{\mathbb R^N}|u|^{4^*}dx)^\frac{2}{4^*}}.
\end{equation}
The lack of compactness in the Sobolev embedding presents a major difficulty in the variational analysis. To recover compactness for subcritical nonlinearities, we restrict the energy functional to the radial subspace and define
\[
S_r(c):=S(c)\cap H_r^2(\mathbb R^N).
\]
The radial embedding \(H_r^2(\mathbb R^N)\hookrightarrow L^t(\mathbb R^N)\)
is compact for every \(2<t<4^*\). However, compactness is not guaranteed at the endpoints \(t=2\) and \(t=4^*\). These cases require additional arguments in the subsequent analysis. We also introduce the radial Poho\v zaev set
$$
\mathcal P_r(c):=\mathcal P(c)\cap S_r(c).
$$
To construct the mountain pass framework, we introduce the auxiliary spaces
\[
E:=H^2(\mathbb R^N)\times\mathbb R,
\qquad
E_r:=H_r^2(\mathbb R^N)\times\mathbb R,
\]
endowed with the norm
\[
\|(u,s)\|_E
:=\left(\|u\|_{H^2}^2+|s|^2_\mathbb R\right)^\frac{1}{2}.
\]
The corresponding norm on \(E_r\) is defined by restriction. We also consider the mass preserving scaling
\[
\mathcal H(u,s)(x):=e^\frac{Ns}{2}u(e^s x),
\]
which preserves the \(L^2\) norm and radial symmetry. In particular, the variational construction in Section \ref{sec:proof-main-theorem-3-and-4} is carried out on \(S_r(c)\times\mathbb R\).
Throughout the paper, \(C\) and \(C_i\) denote positive constants whose values may change from line to line.
\section{Proof of Theorem \ref{Thm:normalized-Subcritical-solutions-existence}}\label{sec:proof-main-theorem-1}

In the first place, assume that
\begin{equation}\label{eqn:Defn-c*}
0<c<\Big(\frac{1}{2\mathcal E}\Big)^\frac{2({4^*-p\gamma_p})}{p(1-\gamma_p)(4^*-2)}=:c_*,
\end{equation}
where
\begin{equation}\label{eqn:A-define}
\mathcal E:=\frac{4^*-p\gamma_p}{2-p\gamma_p}\left(\frac{(2-p\gamma_p)\mu C^p_{N,p}}{(4^*-2)p}\right)^{\frac{4^*-2}{4^*-p\gamma_p}}\left(4^*S^\frac{4^*}{2}\right)^{\frac{p\gamma_p-2}{4^*-p\gamma_p}}.
\end{equation}
To analyze the local variational structure of the energy functional, we introduce the following auxiliary function:
\begin{equation}\label{eqn:Defn-fcr}
g_c(r):=\frac{1}{2}-\frac{\mu}{p}C^p_{N,p}c^\frac{p(1-\gamma_p)}{2}r^{p\gamma_p-2}-\frac{1}{4^*S^\frac{4^*}{2}}r^{4^*-2},\ \ \  r> 0.
\end{equation}
\begin{lemma}\label{Lem:fcr-maximum}
For fixed $c>0$, $g_c(r)$ has a unique global maximum point $r_c$ and the maximum value $g_c(r_c)$ satisfies
$$
\left\{
  \begin{array}{ll}
    \max_{r>0} g_c(r)>0, \ \ & if\ \ \  c< c_*; \\[0.15cm]
    \max_{r>0} g_c(r)=0,\ \  & if\ \ \  c=c_*; \\[0.15cm]
    \max_{r>0} g_c(r)<0,\ \  & if\ \ \  c>c_*,
  \end{array}
\right.
$$
where $c_*$ is defined in (\ref{eqn:Defn-c*}).
\end{lemma}
\begin{proof}
Noting that $0<p\gamma_p <2$ for $2<p<2+\frac{4}{N}$ and $4^*>2$, one has
$$
\lim\limits_{r\to 0^+}g_c(r)=-\infty \ \ \mbox{and}\ \  \lim\limits_{r\to +\infty}g_c(r)=-\infty.
$$
To determine the maximum point of \(g_c\), we differentiate (\ref{eqn:Defn-fcr}) and obtain
$$
g'_c(r)=-\frac{\mu}{p}C^p_{N,p}c^\frac{p(1-\gamma_p)}{2}(p\gamma_p-2)r^{p\gamma_p-3}-\frac{1}{4^*S^\frac{4^*}{2}}(4^*-2)r^{4^*-3},
$$
which means that $g'_c(r)=0$ has a unique solution
\begin{equation}\label{eqn:rc-define}
r_c:=\left(\frac{(2-p\gamma_p)\mu 4^*S^\frac{4^*}{2} C^p_{N,p}c^\frac{p(1-\gamma_p)}{2}}{(4^*-2)p}\right)^\frac{1}{4^*-{p\gamma_p}}.
\end{equation}
Meanwhile, $g_c(r)$ is increasing in $(0,r_c)$ and is decreasing in $(r_c,\infty)$. Hence, the maximum value of $g_c(r)$ is

$$
\begin{aligned}
\max\limits_{r>0}g_c(r)&=g_c(r_c)
=\frac{1}{2}-c^{\frac{p(1-\gamma_p)(4^*-2)}{2({4^*-p\gamma_p})}}\!\!\left(\frac{4^*-p\gamma_p}{2-p\gamma_p}\right)\!\!\left(\frac{(2-p\gamma_p)\mu C^p_{N,p}}{(4^*-2)p}\right)^{\frac{4^*-2}{4^*-p\gamma_p}}\!\!\!\!\left(4^*S^\frac{4^*}{2}\right)^{\frac{p\gamma_p-2}{4^*-p\gamma_p}}.
\end{aligned}
$$
Based on the definition of $c_*$ in (\ref{eqn:Defn-c*}), the conclusions hold true immediately.
\end{proof}

On the basis of Lemma \ref{Lem:fcr-maximum}, we can naturally introduce the constraint set for our problem
(\ref{eqn:BS-equation-L2-Super+Critical}). Actually, consider the function
\begin{equation}\label{eqn:Defn-gc-fc}
\begin{aligned}
h_c(r):=r^2\Bigl[\frac{1}{2}-\frac{\mu}{p}C^p_{N,p}c^\frac{p(1-\gamma_p)}{2}r^{p\gamma_p-2}-\frac{1}{4^*S^\frac{4^*}{2}}r^{4^*-2}\Bigr],\ \ r\geq 0.
\end{aligned}
\end{equation}
For \(r>0\), \(h_c(r)=r^2g_c(r)\) and \(h_c(0)=0\). For simplicity, set
\begin{equation}\label{eqn:Defn-r0}
r_*:= r_{c_*} > 0,
\end{equation}
where \( r_{c_*}\) is determined by (\ref{eqn:rc-define}). Obviously, as a special case of Lemma \ref{Lem:fcr-maximum}, there hold that $g_{c_*}(r_*)=0$ and $g_c(r_*)>0$ for $c\in(0, c_*)$.
Meanwhile, if $c\in(0,c_*)$, $h_c(r)$ has exactly three zeros $0$, $r_1$ and $r_2$ with $0<r_1<r_*<r_2<+\infty$. The assertion about the zeros follows from Lemma \ref{Lem:fcr-maximum} and \(h_c(r)=r^2g_c(r)\) for \(r > 0\). For  \(r > 0\), the
first three derivatives of \(h_c\) are
$$
h'_c(r)= r-{\mu C^p_{N,p}\gamma_pc^\frac{p(1-\gamma_p)}{2}}r^{p\gamma_p-1}-\frac{1}{S^{\frac{4^*}{2}}}r^{4^*-1},
$$
$$
h''_c(r)=1-{\mu C^p_{N,p}\gamma_p(p\gamma_p-1)c^\frac{p(1-\gamma_p)}{2}}r^{p\gamma_p-2}-\frac{4^*-1}{S^{\frac{4^*}{2}}}r^{4^*-2}
$$
and
$$
h'''_c(r)=-{\mu C^p_{N,p}\gamma_p(p\gamma_p-1)(p\gamma_p-2)c^\frac{p(1-\gamma_p)}{2}}r^{p\gamma_p-3}-\frac{(4^*-1)(4^*-2)}{S^{\frac{4^*}{2}}}r^{4^*-3}.
$$
Summarizing the above discussions, we are allowed to introduce the following set:
\begin{equation}\label{eqn:Vr0-define}
V_{r}(c):=\left\{u\in S_r(c):\|\Delta u\|^2_2+\|\nabla u\|^2_2<  r^2_*\right\}.
\end{equation}
In the sequel, the functional $I$ will be restricted to $V_{r}(c)$ for $c\in (0,c_*)$.
\begin{lemma} \label{Lem:Ju-bound}
The functional $I$ is bounded from below in $V_{r}(c)$ and $m_*(c):=\inf_{u\in V_{r}(c)}I(u)<0$.
\end{lemma}
\begin{proof}
By (\ref{eqn:GNinequality}) and (\ref{eqn:Defn-S}), for any $u\in V_{r}(c)$, we obtain that
$$
\begin{aligned}
I(u)\geq h_{c}\Big(\sqrt{\|\Delta u\|^2_2+\|\nabla u\|^2_2}\Big) \geq\inf_{r\in[0,{r_*})}h_{c}(r)>-\infty,
\end{aligned}
$$
which implies that the functional $I$ is bounded from below in $V_{r}(c)$. Fix $u \in S_r(c)$, obviously, $\mathcal{H}(u,s)(x)=e^{\frac{Ns}{2}}u(e^sx) \in S_r(c)$ for $s\in\mathbb R$ and a direct calculation gives that
$$
\begin{aligned}
I(\mathcal{H}(u,s)) & = \frac{e^{4s}}{2}\|\Delta u\|^2_2+\frac{e^{2s}}{2}\|\nabla u\|^2_2-\frac{\mu e^{2p\gamma_p s}}{p}\|u\|^p_p-\frac{e^{24^* s}}{4^*}\|u\|^{4^*}_{4^*}.
\end{aligned}
$$
Since $2p\gamma_p=\frac{N(p-2)}{2}<2$ and $24^*>4$, there exists $s_0<0$ such that $\|\Delta \mathcal{H}(u,s_0)\|^2_2+\|\nabla \mathcal{H}(u,s_0)\|^2_2=e^{4s_0}\|\Delta u\|^2_2+e^{2s_0}\|\nabla u\|^2_2 < r^2_*$ and $I(\mathcal{H}(u,s_0))<0$. Hence, $m_*(c)<0$.
\end{proof}
Next, we prove that $I$ satisfies the $(PS)_a$ condition on $V_{r}(c)$.
\begin{lemma}\label{Lem:local-ps-satisfies}
Let \(0<c<c^*\). Then the restriction of \(I\) to \(V_r(c)\) satisfies the Palais-Smale condition at every negative level \(a<0\).
\end{lemma}
\begin{proof}
Let \(\{u_n\}\subset V_r(c)\) be a Palais-Smale sequence for the restriction of \(I\) to \(V_r(c)\) at a level \(a<0\). Thus,
\begin{equation}\label{eqn:Iun-a-0}
I(u_n)\to a<0,\qquad
(I|_{V_r(c)})'(u_n)\to0
\quad\text{as }n\to\infty.
\end{equation}
By the definition of \(V_r(c)\), we have \(\|\Delta u_n\|_2^2+\|\nabla u_n\|_2^2<r_*^2\) and \(\|u_n\|_2^2=c\). Consequently, \(\{u_n\}\) is bounded in \(H_r^2(\mathbb R^N)\). Since the embedding
\(H_r^2(\mathbb R^N)\hookrightarrow L^p(\mathbb R^N)\) is compact for \(2<p<4^*\), there exist \(u\in H_r^2(\mathbb R^N)\) and a subsequence, still denoted by \(\{u_n\}\), such that
\begin{equation}\label{eqn:bounded-three-proposition}
\begin{cases}
	u_n\rightharpoonup u &\text{in }H_r^2(\mathbb R^N),\\
	u_n\to u &\text{in }L^p(\mathbb R^N),\\
	u_n\to u &\text{a.e. in }\mathbb R^N.
\end{cases}
\end{equation}
Now, we claim that $u\neq 0$. If not, (\ref{eqn:bounded-three-proposition}) brings that $u_n\to 0$ in $L^p({\mathbb R^N})$. Let
$$
h(r):=\frac{1}{2}r^2-\frac{1}{4^*S^\frac{4^*}{2}}r^{4^*},\ \ r\geq 0.
$$
By a direct calculation, we get \(h'(r)=r-\frac{1}{S^\frac{4^*}{2}}r^{4^*-1}\), which implies that $h'(r)=0$ has a unique positive solution \(\tilde{r}:=S^\frac{N}{8}\). Moreover, $h(r)$ is increasing in $(0,\tilde{r})$ and is decreasing in $(\tilde{r},\infty)$. 
Meanwhile, in consideration of (\ref{eqn:Defn-c*}), we see that
\begin{equation}\label{eqn:remove-01}
c^{\frac{p(1-\gamma_p)(4^*-2)}{2(4^*-p\gamma_p)}}\leq\frac{1}{2}\left(\frac{p}{\mu C^p_{N,p}}\right)^{\frac{4^*-2}{4^*-p\gamma_p}}\left(4^*{S^{\frac{4^*}{2}}}\right)^{\frac{2-p\gamma_p}{4^*-p\gamma_p}}\left(\frac{4^*-2}{2-p\gamma_p}\right)^{\frac{4^*-2}{4^*-p\gamma_p}}\frac{2-p\gamma_p}{4^*-p\gamma_p},
\end{equation}
which brings that
\begin{equation}\label{eqn:remove-02}
\begin{aligned}
c^{\frac{p(1-\gamma_p)}{4^*-p\gamma_p}}\leq r^{-2}_cc^\frac{p(1-\gamma_p)}{4^*-p\gamma_p}\left(\frac{(2-p\gamma_p)4^*S^{\frac{4^*}{2}}}{2(4^*-p\gamma_p)}\right)^\frac{2}{4^*-2}.\\
\end{aligned}
\end{equation}
Equivalently, it states that
\begin{equation}\label{eqn:remove-03}
\begin{aligned}
r_c\leq\left(\frac{(2-p\gamma_p)4^*S^{\frac{4^*}{2}}}{2(4^*-p\gamma_p)}\right)^\frac{1}{4^*-2}=\left(\frac{8N-N^2(p-2)}{8N-N^2(p-2)+4N(p-2)}\right)^\frac{1}{4^*-2}S^\frac{N}{8}<S^\frac{N}{8}=\tilde{r}.
\end{aligned}
\end{equation}
In particular, there holds that
\begin{equation}\label{eqn:r0-gamma-s-relation}
0<r_*=r_{c_*}<S^\frac{N}{8}=\tilde{r}.
\end{equation}
Since $h(0)=0$ and $h$ is increasing on $[0, r_*)$, we obtain that
\begin{equation}\label{eqn:h-0-r-r0}
h(r)\geq 0,\ \ \ r\in[0,r_*).
\end{equation}
Hence, utilizing (\ref{eqn:Defn-Iu}), (\ref{eqn:Defn-S}), (\ref{eqn:h-0-r-r0}) and $\|u_n\|^p_p\to 0$, we obtain that
\begin{equation}\label{eqn:lambda-leq-0-strict-L2sub}
\begin{aligned}
I(u_n)&\geq -\frac{\mu}{p}\|u_n\|^p_p=o_n(1),
\end{aligned}
\end{equation}
which contradicts $\lim_{n\to\infty}I(u_n)=a<0$ in (\ref{eqn:Iun-a-0}). Thus, the weak limit $u$ is nontrivial.

For the sequence $\{u_n\}$, in view of \cite[Proposition 5.12]{Willem1996}, there exists $\{\lambda_n\}\subset \mathbb R$ such that
\begin{equation}\label{eqn:J'un-lambdapsiun-0}
I'(u_n)-\lambda_n\Psi'(u_n)\to 0\ \ \mbox{in}\ \  H^{-2}_{r}(\mathbb R^N)\ \ \mbox{as}\ \ n\to\infty,
\end{equation}
where $H^{-2}_{r}(\mathbb R^N)$ is the dual space of $H^{2}_{r}(\mathbb R^N)$ and $\Psi:H^2_{r}(\mathbb{R}^N)\to \mathbb{R}$ is given by
\begin{equation}\label{eqn:Psi-define}
\Psi(u)=\frac{1}{2}\int_{\mathbb{R}^N}|u|^2dx.
\end{equation}
It follows from (\ref{eqn:J'un-lambdapsiun-0}) that there exists a sequence \(\varepsilon_n\to0\) such that, for every \(\varphi\in H_r^2(\mathbb R^N)\),
\begin{equation}\label{eqn:Jun-varphi}
\left|I'(u_n)[\varphi]-\lambda_n\int_{\mathbb R^N}u_n\varphi dx\right|\leq\varepsilon_n\|\varphi\|_{H^2}.
\end{equation}
Hence, due to the boundedness of $\{u_n\}$ in $H^2_{r}(\mathbb R^N)$, we see that
\begin{equation}\label{eqn:Iun-un-on1-L2sub}
\|\Delta u_n\|^2_2+\|\nabla u_n\|^2_2-\mu\|u_n\|^p_p-\|u_n\|^{4^*}_{4^*}=\lambda_n\|u_n\|^2_2+o_n(1)=\lambda_n c+o_n(1),
\end{equation}
which indicates that $\{\lambda_n\}$ is a bounded sequence, and thus, up to a subsequence, there exists $\lambda\in\mathbb R$ such that $\lambda_n\to \lambda$ as $n\to \infty$. Now, we claim $\lambda<0$. Indeed, it follows from (\ref{eqn:Iun-un-on1-L2sub}) and the fact $\lim_{n\to\infty}I(u_n)=a<0$ that
$$
\begin{aligned}
\lambda_n c=\lambda_n\|u_n\|^2_2
=2a+\frac{(2-p)\mu}{p}\|u_n\|^p_p+\frac{2-4^*}{4^*}\|u_n\|^{4^*}_{4^*}+o_n(1)\leq 2a+o_n(1).
\end{aligned}
$$
Passing to the limit as $n\to \infty$ in both sides of the above inequality, we have $\lambda\leq \frac{2a}{c}<0$.

Denote by $\mathcal M(\mathbb R^N)$ the space of finite measures in $\mathbb R^N$. Using the concentration-compactness principle due to \cite[Lemma 2.1]{Alves-doO2002}, we know that there exist two bounded positive measures $\kappa,\nu\in \mathcal M(\mathbb R^N)$ such that \(|\Delta u_n|^2dx\rightharpoonup \kappa\) and \(|u_n|^{4^*}dx\rightharpoonup\nu\) in \(\mathcal M(\mathbb R^N)\) as \(n\to\infty\). Then there exist an at most countable index set $\mathcal J$, a family $\{x_j\}_{j\in\mathcal J}$ of distinct points in $\mathbb R^N$ and a family $\{\nu_j\}_{j\in\mathcal J}$ of positive numbers such that \(\nu=|u|^{4^*}dx+\sum_{j \in \mathcal J} \nu_j \delta_{x_j}\), where $\delta_{x_j}$ is the Dirac mass at the point $x_j$. Moreover, we have \(\kappa \geq |\Delta u|^2dx + \sum_{j \in \mathcal J} \kappa_j \delta_{x_j}\) for some family $\{\kappa_j\}_{j \in \mathcal J}$ of positive numbers satisfying
\begin{equation}\label{eqn:S-nuj-kappa-j}
S{\nu_j}^\frac{2}{4^*} \leq \kappa_j, \ \ \  j\in \mathcal J.
\end{equation}
In particular, $\sum_{j \in \mathcal J} \nu_j^\frac{2}{4^*}<\infty$. Since each $u_n$ is radially symmetric, the measures $|u_n|^{4^*}dx$ are invariant under rotations, and their weak limit $\nu$ has the same invariance. A finite rotation-invariant measure cannot have a positive atom away from the origin: rotations of any nonzero point would generate infinitely many distinct atoms of the same positive mass. Thus every concentration point must be the origin. If $\mathcal J$ is nonempty, it consists of a single index, which we denote by $0$, with $x_0 = 0$. Denote the corresponding atomic masses of \(\nu\) and \(\kappa\) by \(\nu_0\) and \(\kappa_0\), respectively. To examine the possible concentration at the origin, choose a radial function \(\eta\in C_0^\infty(\mathbb R^N,[0,1])\) such that
\[
\eta(x)=
\begin{cases}
	1,& |x|\leq \frac{1}{2};\\
	0,& |x|\geq1.
\end{cases}
\]
For \(0<\rho<1\), define \(\varphi_\rho(x):=\eta\left(\frac{x}{\rho}\right)\).
Then \(\varphi_\rho\) is radial and satisfies
\[
\varphi_\rho(x)=
\begin{cases}
	1,& |x|\leq \frac{\rho}{2};\\
	0,& |x|\geq\rho.
\end{cases}
\]
Moreover, \(|\nabla\varphi_\rho|\leq\frac C\rho\) and \(|\Delta\varphi_\rho|\leq\frac C{\rho^2}\) (see also \cite[Claim 2.3]{Alves-doO2002}). In particular, \(\varphi_\rho u_n\in H_r^2(\mathbb R^N)\) for every \(n\). Thus, it is an admissible test function in (\ref{eqn:J'un-lambdapsiun-0}). By the local compact embedding \(H^2_{{loc}}(\mathbb R^N)\hookrightarrow H^1_{{loc}}(\mathbb R^N)\), after passing to a subsequence, we also have \(u_n\to u\) in \(H^1_{{loc}}(\mathbb R^N)\). Furthermore, the strong convergence in (\ref{eqn:bounded-three-proposition}) implies that \(u_n\to u\) in \(L^2(B_\rho(0))\) for every fixed \(\rho>0\). Now, by (\ref{eqn:J'un-lambdapsiun-0}), there exists a sequence \(\varepsilon_n\to0\) such that
\[
\left|I'(u_n)[\varphi_\rho u_n]-\lambda_n\int_{\mathbb R^N}\varphi_\rho|u_n|^2\,dx\right|
\leq\varepsilon_n\|\varphi_\rho u_n\|_{H^2}=o_n(1)
\]
for each fixed \(\rho>0\). Using the definition of \(I\), we obtain
\begin{equation}\label{eqn:3-21}
\begin{aligned}
\int_{\mathbb R^N}|u_n|^{4^*}\varphi_\rho\,dx&=
\int_{\mathbb R^N}|\Delta u_n|^2\varphi_\rho\,dx
+2\int_{\mathbb R^N}\Delta u_n\nabla\varphi_\rho\cdot\nabla u_n\,dx
+\int_{\mathbb R^N}u_n\Delta u_n\Delta\varphi_\rho\,dx\\
&\quad+\int_{\mathbb R^N}|\nabla u_n|^2\varphi_\rho\,dx+\int_{\mathbb R^N}u_n\nabla u_n\cdot\nabla\varphi_\rho\,dx-\lambda_n\int_{\mathbb R^N}
	|u_n|^2\varphi_\rho\,dx\\
&\quad-\mu\int_{\mathbb R^N}|u_n|^p\varphi_\rho\,dx+o_n(1).
\end{aligned}
\end{equation}
It remains to control the three terms involving derivatives of the cut-off function. Let \(A_\rho:=B_\rho(0)\setminus B_{\frac{\rho}{2}}(0)\). Since the derivatives of \(\varphi_\rho\) vanish outside \(A_\rho\), H\"older's inequality and the boundedness of \(\{u_n\}\) in \(H^2(\mathbb R^N)\) yield
\[
\begin{aligned}
\left|\int_{\mathbb R^N}u_n\Delta u_n\Delta\varphi_\rho\,dx\right|
&\leq\|\Delta u_n\|_2\left(\int_{A_\rho}|u_n|^2|\Delta\varphi_\rho|^2\,dx\right)^\frac{1}{2}
\leq C\rho^{-2}\left(\int_{B_\rho}|u_n|^2\,dx\right)^\frac{1}{2}.
\end{aligned}
\]
For each fixed \(\rho>0\), passing to the limit as \(n\to\infty\) gives
\[
\begin{aligned}
\limsup_{n\to\infty}\left|\int_{\mathbb R^N}u_n\Delta u_n\Delta\varphi_\rho\,dx\right|
&\leq C\rho^{-2}\left(\int_{B_\rho}|u|^2\,dx\right)^\frac{1}{2}
\leq C\left(\int_{B_\rho}|u|^{4^*}\,dx\right)^\frac{1}{4^*}.
\end{aligned}
\]
Here we used \(\left(\int_{B_\rho}|u|^2\,dx\right)^\frac{1}{2} \leq C\rho^2\left(\int_{B_\rho}|u|^{4^*}\,dx\right)^\frac{1}{4^*}\).
Since \(u\in L^{4^*}(\mathbb R^N)\), the last expression tends to zero as \(\rho\downarrow0\). Hence,
\[
\lim_{\rho\downarrow0}\limsup_{n\to\infty}
\left|\int_{\mathbb R^N}u_n\Delta u_n\Delta\varphi_\rho\,dx\right|=0.
\]
Similarly, we obtain
\[
\begin{aligned}
\left|\int_{\mathbb R^N}u_n\nabla u_n\cdot\nabla\varphi_\rho\,dx\right|
&\leq\|\nabla u_n\|_2\left(\int_{A_\rho}|u_n|^2|\nabla\varphi_\rho|^2\,dx\right)^\frac{1}{2}
\leq C\rho^{-1}\left(\int_{B_\rho}|u_n|^2\,dx\right)^\frac{1}{2}.
\end{aligned}
\]
Consequently,
\[
\limsup_{n\to\infty}\left|\int_{\mathbb R^N}u_n\nabla u_n\cdot\nabla\varphi_\rho\,dx\right|
\leq C\rho^{-1} \left(\int_{B_\rho}|u|^2\,dx\right)^\frac{1}{2}
\leq C\rho\left(\int_{B_\rho}|u|^{4^*}\,dx\right)^\frac{1}{4^*}\rightarrow0
\]
as \(\rho\downarrow0\). For the remaining mixed term, the Cauchy-Schwarz inequality gives
\[
\begin{aligned}
\left|\int_{\mathbb R^N}\Delta u_n\nabla\varphi_\rho\cdot\nabla u_n\,dx\right|
&\leq\|\Delta u_n\|_2\left(\int_{B_\rho}|\nabla u_n|^2|\nabla\varphi_\rho|^2\,dx\right)^\frac{1}{2}
\leq C\rho^{-1}\left(\int_{B_\rho}|\nabla u_n|^2\,dx\right)^\frac{1}{2}.
\end{aligned}
\]
Since \(\nabla u_n\to\nabla u\) in \(L^2(B_\rho)\), it follows that
\[
\begin{aligned}
\limsup_{n\to\infty}\left|\int_{\mathbb R^N}\Delta u_n\nabla\varphi_\rho\cdot\nabla u_n\,dx\right|
&\leq C\rho^{-1}\left(\int_{B_\rho}|\nabla u|^2\,dx\right)^\frac{1}{2}
\leq C\left(\int_{B_\rho}|\nabla u|^{\frac{2N}{N-2}}\,dx\right)^{\frac{N-2}{2N}}.
\end{aligned}
\]
The last expression converges to zero as \(\rho\downarrow0\), because \(\nabla u\in L^\frac{2N}{N-2}(\mathbb R^N)\). Therefore, all three derivative terms vanish in the iterated limit. On the other hand, the strong convergence \(u_n\to u\) in \(L^p(\mathbb R^N)\) and the convergence \(\lambda_n\to\lambda\) imply that, for every fixed \(\rho>0\),
\[
\int_{\mathbb R^N}|u_n|^p\varphi_\rho\,dx
\rightarrow
\int_{\mathbb R^N}|u|^p\varphi_\rho\,dx
\]
and
\[
\lambda_n\int_{\mathbb R^N}|u_n|^2\varphi_\rho\,dx
\rightarrow
\lambda\int_{\mathbb R^N}|u|^2\varphi_\rho\,dx.
\]
Since \(0\leq\varphi_\rho\leq1\) and its support is contained in \(B_\rho(0)\), the absolute continuity of the Lebesgue integral gives
\[
\lim_{\rho\downarrow0}\lim_{n\to\infty}
\mu\int_{\mathbb R^N}|u_n|^p\varphi_\rho\,dx=0
\]
and
\[
\lim_{\rho\downarrow0}\lim_{n\to\infty}
\lambda_n\int_{\mathbb R^N}|u_n|^2\varphi_\rho\,dx=0.
\]
Combining these estimates with (\ref{eqn:3-21}), we obtain
\[
\begin{aligned}
\nu_0&=\lim_{\rho\downarrow0}\lim_{n\to\infty}\int_{\mathbb R^N}|u_n|^{4^*}\varphi_\rho\,dx
\geq\lim_{\rho\downarrow0}\lim_{n\to\infty}\int_{\mathbb R^N}|\Delta u_n|^2\varphi_\rho\,dx
=\lim_{\rho\downarrow0}\int_{\mathbb R^N}\varphi_\rho\,d\kappa
=\kappa_0.
\end{aligned}
\]
The last equality follows from the fact that \(\varphi_\rho\) equals one near the origin and its support shrinks to the origin as \(\rho\downarrow0\). Thus, \(\nu_0\geq\kappa_0\). Together with (\ref{eqn:S-nuj-kappa-j}), this yields \(S\nu_0^{\frac2{4^*}}\leq\kappa_0\leq\nu_0\). Since \(\nu_0>0\) and \(1-\frac{2}{4^*}=\frac{4}{N}\), we deduce that \(\nu_0\geq S^\frac{N}{4}\) and \(\kappa_0\geq S^\frac{N}{4}\). However, the definition of \(V_r(c)\) gives
\(r_*^2\geq\liminf_{n\to\infty}\left(\|\Delta u_n\|_2^2+\|\nabla u_n\|_2^2\right)
\geq\kappa_0\geq S^\frac{N}{4}\), which contradicts (\ref{eqn:r0-gamma-s-relation}), namely \(r_*^2<S^\frac{N}{4}\). Therefore, the assumption \(\mathcal J\neq\varnothing\) is false, and hence \(\mathcal J=\varnothing\). Consequently, the concentration measure reduces to \(\nu=|u|^{4^*}\,dx\). For every \(R>0\), the weak convergence of measures and the fact that \(\partial B_R(0)\) has zero \(\nu\)-measure imply \(\int_{B_R(0)}|u_n|^{4^*}\,dx\rightarrow\int_{B_R(0)}|u|^{4^*}\,dx\).
Together with the weak convergence in \(L^{4^*}(B_R(0))\) and the uniform convexity of this space, we conclude that
\begin{equation}\label{eqn:un-u-loc-4^*}
u_n\to u\quad\text{in }L_{{loc}}^{4^*}(\mathbb R^N).
\end{equation}

On the other hand, since $\{u_n\}$ is bounded in $H^2_{r}(\mathbb R^N)$ 
and \(H^2_{{r}}(\mathbb R^N)\hookrightarrow H^1(\mathbb R^N)\), the Radial Lemma A.II in \cite{Berestycki-Lions1983} gives
$$
|u_n(x)|\leq \frac{C_1\|u_n\|_{H^2}}{|x|^\frac{N-1}{2}}\leq \frac{C_2}{|x|^\frac{N-1}{2}}\ \ \ a.e.\ \ \ \mbox{in}\ \mathbb R^N,
$$
so, one has
$$
|u_n(x)|^{4^*}\leq \frac{C_3}{|x|^\frac{N(N-1)}{N-4}}\ \ \ a.e.\ \ \ \mbox{in}\ \mathbb R^N.
$$
Recalling that $\frac{C_3}{|\cdot|^\frac{N(N-1)}{N-4}}\in L^1\!\!\left(\mathbb R^N\setminus B_R(0)\right)$ and $u_n\to u$ a.e. in $\mathbb R^N\setminus B_R(0)$, the Lebesgue dominated convergence theorem gives
\begin{equation}\label{eqn:un-u-4^*}
u_n\to u\ \ \ \mbox{in}\ \ \ L^{4^*}\!\!\left(\mathbb R^N \setminus B_R(0)\right).
\end{equation}
From (\ref{eqn:un-u-loc-4^*}) and (\ref{eqn:un-u-4^*}), we conclude that $u_n\to u$ in $ L^{4^*}(\mathbb R^N)$.

Now, using $u_n\to u$ in $L^{4^*}(\mathbb R^N)$ and $u_n\to u$ in $L^p(\mathbb R^N)$, we obtain
$$
\lim\limits_{n\to\infty}\int_{\mathbb R^N}\left(\mu|u_n|^p+|u_n|^{4^*}\right)dx=\int_{\mathbb R^N}\left(\mu|u|^p+|u|^{4^*}\right)dx.
$$
Passing to the limit in (\ref{eqn:Jun-varphi}), using weak convergence in $H^2_r(\mathbb R^N)$ and strong convergence in $L^p(\mathbb R^N)$ and $L^{4^*}(\mathbb R^N)$, we find that $u$ is a nontrivial weak solution of
$$
{\Delta}^2u- \Delta u-\lambda u = \mu |u|^{p-2}u+ |u|^{4^*-2}u\ \ \mbox{in}\ \mathbb{R}^N.
$$
The equation holds for all $H^2(\mathbb R^N)$ test functions by rotational averaging. Testing it with $u$ and using (\ref{eqn:Iun-un-on1-L2sub}), we have
\begin{equation}\label{eqn:un-u-H2rad}
\begin{aligned}
\lim\limits_{n\to\infty}\Big(\|\Delta u_n\|^2_2 +\|\nabla u_n\|^2_2-\lambda\|u_n\|^2_2\Big)&=\lim\limits_{n\to\infty}\Big(\mu\|u_n\|^p_p+\|u_n\|^{4^*}_{4^*}\Big)
=\|\Delta u\|^2_2 +\|\nabla u\|^2_2-\lambda\|u\|^2_2.
\end{aligned}
\end{equation}
As $\lambda<0$, (\ref{eqn:un-u-H2rad}) yields that
$$
\begin{aligned}
	-\lambda\|u\|_2^2 & \leq \liminf _{n \rightarrow \infty}\Big(-\lambda\|u_n\|_2^2\Big) \\
	&\leq \limsup _{n \rightarrow \infty}\Big(-\lambda\|u_n\|_2^2\Big) \\
	& \leq \limsup _{n \rightarrow \infty}\Big(-\lambda\|u_n\|_2^2\Big)+\liminf _{n \rightarrow \infty}\Big(\|\Delta u_n\|^2_2+\|\nabla u_n\|^2_2\Big)-\Big(\|\Delta u\|^2_2+\|\nabla u\|^2_2\Big)\\
	& \leq \limsup _{n \rightarrow \infty}\Big(\|\Delta u_n\|^2_2+\|\nabla u_n\|^2_2-\lambda\|u_n\|_2^2\Big)-\Big(\|\Delta u\|^2_2+\|\nabla u\|^2_2\Big)
	=-\lambda\|u\|_2^2,
\end{aligned}
$$
and so $\|u_n\|^2_2\to \|u\|^2_2=c$ as $n\to \infty$. Together this with (\ref{eqn:un-u-H2rad}) again, we conclude that
$$
\|\Delta u_n\|^2_2+\|\nabla u_n\|^2_2\to \|\Delta u\|^2_2+\|\nabla u\|^2_2\ \ \mbox{as}\ \ n\to \infty.
$$
Hence, $u_n\to u$ in $H^2_{r}(\mathbb R^N)$.
\end{proof}

\textbf{Proof of Theorem \ref{Thm:normalized-Subcritical-solutions-existence}.}
Let \(m^*(c):=\inf_{u\in V_r(c)}I(u)<0\). By Lemma \ref{Lem:Ju-bound}, the functional \(I\) is bounded from below on \(V_r(c)\). Moreover, (\ref{eqn:Defn-gc-fc})–(\ref{eqn:Vr0-define}) imply that \(I(u)\geq h_c(r_*)>0\) for every \(u\in S_r(c)\) satisfying \(\|\Delta u\|_2^2+\|\nabla u\|_2^2=r_*^2\). Thus, the negative infimum \(m^*(c)\) lies strictly below the energy on the boundary of \(V_r(c)\). 
Apply Ekeland's variational principle (see \cite[Section 2.2]{Willem1996}) to the closure of \(V_r(c)\) in \(S_r(c)\), which is complete in the \(H^2\) metric. Since \(m^*(c)<0\)  and the boundary energy is strictly positive, the resulting minimizing sequence lies in the interior for all sufficiently large indices and satisfies \((I|_{S_r(c)})'(u_n)\to 0\).
Thus, it is a minimizing Palais-Smale sequence in \(V_r(c)\) at the level \(m^*(c)\).
By Lemma \ref{Lem:local-ps-satisfies}, a subsequence converges strongly to some \(u\in H_r^2(\mathbb R^N)\). Since \(u_n\to u\) strongly in \(H_r^2(\mathbb R^N)\) and \(u_n\in V_r(c)\), we have \(\|\Delta u\|_2^2+\|\nabla u\|_2^2\leq r_*^2\). If equality held, the boundary estimate would give \(I(u)\geq h_c(r_*)>0\). However, \(I(u)=m^*(c)<0\). This contradiction shows that \(\|\Delta u\|_2^2+\|\nabla u\|_2^2<r_*^2\). Therefore, \(u\in V_r(c)\). Consequently, \(I(u)=m^*(c)<0\). Since \(V_r(c)\) is open in \(S_r(c)\), the function \(u\) is a constrained critical point of \(I\) on \(S_r(c)\). Furthermore, both the energy functional and the mass constraint are invariant under rotations. The principle of symmetric criticality therefore implies that \(u\) is also a constrained critical point of \(I\) on \(S(c)\). Hence, \(u\) solves (\ref{eqn:BS-equation-L2-Super+Critical}) for some Lagrange multiplier \(\lambda\in\mathbb R\). Finally, the argument in Lemma \ref{Lem:local-ps-satisfies} shows that \(\lambda<0\). This completes the proof.
\\
\qed

\section{Proof of Theorem \ref{Thm:normalized-Subcritical-solutions-infinity}}\label{sec:proof-main-theorem-2}

Now, we present a minimax theorem which is proved by Jeanjean and Lu \cite[Section 2]{Jeanjean-Lu2019}. In order to formulate the minimax theorem, some notations are needed. Let $\mathbb X$ be a real Banach space with norm $\|\cdot\|_\mathbb X$ and $\mathbb H$ be a real Hilbert space with inner product $(\cdot,\cdot)_{\mathbb H}$. In the sequel, let us identify $\mathbb H$ with its dual space and assume that $\mathbb X$ is embedded continuously in $\mathbb H$. For any $k>0$, define the manifold
$$
\mathcal Q:=\Big\{u\in \mathbb X\ |\ (u,u)_{\mathbb H}= k\Big\},
$$
which is endowed with the topology inherited from $\mathbb X$.
Clearly, the tangent space of $\mathcal Q$ at a point $u\in\mathcal Q$ is defined by
$$
T_u\mathcal Q:=\Big\{v\in \mathbb X\ |\ (u,v)_{\mathbb H}=0\Big\}.
$$
Let $J\in C^1(\mathbb X,\mathbb R)$, then $J|_{\mathcal Q}$ is a functional of class $C^1$ on $\mathcal Q$. The norm of the derivative of
$J|_{\mathcal Q}$ at any point $u\in \mathcal Q$ is defined by
$$
\|J'|_{\mathcal Q}(u)\|:=\sup\limits_{\|v\|_\mathbb X\leq 1,\ v\in T_u\mathcal Q}\left|\langle J'(u),v \rangle\right|.
$$
A point $u\in \mathcal Q$ is said to be a critical point of $J|_{\mathcal Q}$ if $\left(J|_{\mathcal Q}\right)'(u)=0$ (or, equivalently, $\|\left(J|_{\mathcal Q}\right)'(u)\|=0$). A number $a\in\mathbb R$ is called a critical value of $J|_{\mathcal Q}$ if $J|_{\mathcal Q}$ has a critical point $u\in\mathcal Q$ such that $a = J(u)$.
Noting that $\mathcal Q$ is symmetric with respect to $0\in\mathbb X$ and $0\notin \mathcal Q$, we introduce the notation of the genus. Let $\Sigma$ be the family of closed symmetric subsets of $\mathcal Q$. For any nonempty set \(\Gamma\in\Sigma\), the genus \(\mathcal G(\Gamma)\) of \(\Gamma\) is defined as the least integer \(m\geq1\) for which there exists an odd continuous mapping \(\varphi:\Gamma\to\mathbb R^m\setminus\{0\}\). We set \(\mathcal G(\Gamma)=+\infty\) if no such integer exists and \(\mathcal G(\Gamma)=0\) if \(\Gamma=\varnothing\). For each $m\in \mathbb N^+$, let
$$
\Gamma_m:=\Big\{\Gamma\in\Sigma\ |\ \mathcal G(\Gamma)\geq m\Big\}.
$$
We now state the minimax theorem used below, which is a particular case of \cite[Theorem 2.1]{Jeanjean-Lu2019}.
\begin{proposition} \label{Pro:minimax-theorem}
Let $J:\mathbb X\to\mathbb R$ be an even functional of class $C^1$. Assume that $J|_{\mathcal Q}$ is bounded from below and satisfies the $(PS)_a$ condition for all $a<0$, and that $\Gamma_m\neq\emptyset$ for each $m\in \mathbb N^+$. Then the minimax values $-\infty<a_1\leq a_2\leq\cdots\leq a_m\leq \cdots$ can be defined as follows:
$$
a_m:=\inf\limits_{\Gamma\in\Gamma_m}\sup\limits_{u\in\Gamma}J(u),\ \ \ m\geq 1,
$$
and the following statements hold:
\begin{itemize}
\item[(i)]  $a_m$ is a critical value of $J|_{\mathcal Q}$ provided $a_m<0$;
\item[(ii)] denote by $K_a$ the set of critical points of $J|_{\mathcal Q}$ at one level $a\in\mathbb R$. If
$$
a_m=a_{m+1}=\cdots=a_{m+l-1}=:a<0\ \ \mbox{for}\ \ m,\ l\geq 1,
$$
then $\mathcal G(K_a)\geq l$. In particular, \(J|_Q\) has infinitely many critical points at the level \(a\) if \(l\geq2\).
\item[(iii)] if $a_m<0$ for all $m\geq 1 $, then $a_m\to 0^-$ as $m\to+\infty$.
\end{itemize}
\end{proposition}

\begin{lemma} \label{Lem:some-properties}
Fix \(\mu>0\) and \(0<c<c^*(\mu)\). For every \(m\in\mathbb N^+\), there exists an \(m\)-dimensional subspace \(E_m\subset H_r^2(\mathbb R^N)\) with a basis \(\{u_1,\ldots,u_m\}\subset V_r(c)\) such that
\[
\int_{\mathbb R^N}u_i u_j\,dx=0,
\quad i\neq j,
\]
and
\[
T_m:=
\left\{\sum_{i=1}^m s_i u_i:s_i\in\mathbb R,\ \sum_{i=1}^m s_i^2=1\right\}\subset V_r(c).
\]
Moreover,
\[
\max_{u\in T_m}I(u)<0.
\]
\end{lemma}
\begin{proof}
Choose \(v_1,\ldots,v_m\in C_0^\infty(\mathbb R^N)\cap H_r^2(\mathbb R^N)\) such that \(\|v_i\|_2^2=c\) and
\[
\operatorname{supp}v_i\cap\operatorname{supp}v_j=\varnothing,
\quad i\neq j.
\]
For \(\tau\leq0\), define
\[
\Phi(\tau,v_i)(x):=e^{\frac N2\tau}v_i(e^\tau x),
\qquad x\in\mathbb R^N.
\]
The mass preserving scaling gives \(\|\Phi(\tau,v_i)\|_2^2=c\) and preserves the disjointness of the supports. Moreover, set
\[
\omega_i(\tau):=e^{4\tau}\|\Delta v_i\|_2^2+e^{2\tau}\|\nabla v_i\|_2^2.
\]
Then
\[
\|\Delta\Phi(\tau,v_i)\|_2^2+\|\nabla\Phi(\tau,v_i)\|_2^2=\omega_i(\tau).
\]
In particular,
\[
\max_{1\leq i\leq m}\omega_i(\tau)
\rightarrow 0
\qquad\text{as }\tau\to-\infty.
\]
Set \(u_i:=\Phi(\tau,v_i)\) and \(E_m:=\operatorname{span}\{u_1,\ldots,u_m\}\). For any \(s_1,\ldots,s_m\in\mathbb R\) satisfying \(\sum_{i=1}^m s_i^2=1\), let \(\widetilde v:=\sum_{i=1}^m s_i u_i\). Since the supports of the \(u_i\)'s are pairwise disjoint, we have
\[
\|\widetilde v\|_2^2=\sum_{i=1}^m s_i^2\|u_i\|_2^2=c
\]
and
\[
\begin{aligned}
\|\Delta\widetilde v\|_2^2+\|\nabla\widetilde v\|_2^2
=\sum_{i=1}^m s_i^2\omega_i(\tau)\leq \max_{1\leq i\leq m}\omega_i(\tau).
\end{aligned}
\]
Consequently, \(T_m\subset V_r(c)\) whenever \(\max_{1\leq i\leq m}\omega_i(\tau)<r_*^2\). Next, write \(C_i:=\|v_i\|_p^p>0\). The disjointness of the supports and the scaling property yield \(\|\widetilde v\|_p^p=e^{2p\gamma_p\tau}\sum_{i=1}^m |s_i|^p C_i\). By the compactness of the unit sphere in \(\mathbb R^m\), the constant \(d_m:=\min_{\sum_{i=1}^m s_i^2=1}\sum_{i=1}^m |s_i|^p C_i\) is strictly positive. Hence,
\[
\|\widetilde v\|_p^p\geq d_m e^{2p\gamma_p\tau}.
\]
Combining these estimates with the definition of \(I\), we obtain
\[
\begin{aligned}
I(\widetilde v)&\leq\frac{1}{2}\max_{1\leq i\leq m}\omega_i(\tau)-\frac{\mu}{p}d_m e^{2p\gamma_p\tau}.
\end{aligned}
\]
Since \(2<p<2+\frac4N\), we have \(2p\gamma_p=\frac{N(p-2)}2<2\). On the other hand,
\[
\max_{1\leq i\leq m}\omega_i(\tau)=O(e^{2\tau})
\quad\text{as }\tau\to-\infty.
\]
Therefore, for sufficiently negative \(\tau\), \(\frac12\max_{1\leq i\leq m}\omega_i(\tau)
-\frac{\mu}{p}d_m e^{2p\gamma_p\tau}<0\). Choosing \(\tau\) sufficiently negative so that both this inequality and the preceding bound by \(r_*^2\) hold, we conclude that \(\max_{u\in T_m}I(u)<0\). This completes the proof.
\end{proof}

To apply Proposition \ref{Pro:minimax-theorem} on the complete mass constraint, we introduce an even truncated functional on \(S_r(c)\). Set \(R(u):=\|\Delta u\|_2^2+\|\nabla u\|_2^2\). As shown in the proof of Lemma \ref{Lem:Ju-bound}, the estimates (\ref{eqn:GNinequality}) and (\ref{eqn:Defn-S}) yield \(I(u)\geq h_c(\sqrt{R(u)})\) for every \(u\in S_r(c)\). Since \(h_c(r_*)>0\), there exists \(\delta>0\) such that \(h_c(\sqrt t)>0\) for \(t\in[r_*^2,r_*^2+\delta]\). Choose a smooth function \(\chi:[0,\infty)\to[0,1]\) satisfying \(\chi(t)=1\) for \(t\leq r_*^2\) and \(\chi(t)=0\) for \(t\geq r_*^2+\delta\). Define
\[
J(u):=\chi(R(u))I(u)+\frac{1-\chi(R(u))}{2}R(u),\quad u\in H_r^2(\mathbb R^N).
\]
Then \(J\) is even and of class \(\mathcal C^1\) on \(H_r^2(\mathbb R^N)\), and \(J=I\) on \(V_r(c)\). We consider its restriction to \(S_r(c)\). By the choice of \(\delta\) and \(\chi\), we have \(J(u)>0\) whenever \(R(u)\geq r_*^2\). Consequently, every negative sublevel of \(J\) is contained in \(V_r(c)\). Furthermore, \(J\) is bounded from below on \(S_r(c)\). Indeed, the lower bound for \(I\) on the bounded region \(R(u)\leq r_*^2+\delta\) follows from \(I(u)\geq h_c(\sqrt{R(u)})\). Outside this region, \(J(u)=\frac{R(u)}{2}\geq0\). We next verify the Palais-Smale condition at negative levels. Let \(\{u_n\}\subset S_r(c)\) satisfy
\[
J(u_n)\to a<0,
\quad
(J|_{S_r(c)})'(u_n)\to0.
\]
For sufficiently large \(n\), the negativity of \(J(u_n)\) implies that \(u_n\in V_r(c)\). Since \(J=I\) in this region, we also have
\[
I(u_n)\to a<0,
\quad
(I|_{V_r(c)})'(u_n)\to0.
\]
By Lemma \ref{Lem:local-ps-satisfies}, the sequence possesses a strongly convergent subsequence in \(H_r^2(\mathbb R^N)\). Therefore, \(J|_{S_r(c)}\) satisfies the Palais-Smale condition at every negative level. 
Let
\[
\Gamma_m:=\{A\in\Sigma:\mathcal G(A)\geq m\},
\]
and define
\[
a_m:=\inf_{A\in\Gamma_m}\sup_{u\in A}J(u).
\]
In the following, we prove that \(\Gamma_m\neq\emptyset\) and \(a_m\in(-\infty,0)\) for every \(m\in\mathbb N^+\).

\begin{lemma}\label{lem:Sigma-neq-emptyset}
For any $m\in \mathbb N^+$, one has ${\Gamma}_m\neq\emptyset$ and $a_m\in(-\infty,0)$.
\end{lemma}
\begin{proof}
Let $T_m\subseteq V_{r}(c)$ be obtained by Lemma \ref{Lem:some-properties} and then, $T_m\in{\Sigma}$. Let \(\mathcal S^{m-1}\) denote the unit sphere in \(\mathbb R^m\). The map \(\Psi:\mathcal S^{m-1}\rightarrow T_m\), \((s_1,\ldots,s_m)\mapsto
\sum_{i=1}^m s_i u_i\) is an odd homeomorphism. Therefore, by the invariance of the genus under odd homeomorphisms and \cite[Corollary 5.3]{Zhang-Zhang2022}, we have \(G(T_m)=G(S^{m-1})=m\). This implies that \(T_m\in\Gamma_m\). Furthermore, since \(J=I\) on \(T_m\subset V_r(c)\), combining Lemmas \ref{Lem:Ju-bound} and \ref{Lem:some-properties} gives
\[
-\infty<m^*(c)\leq a_m\leq\sup_{u\in T_m}J(u)=\sup_{u\in T_m}I(u)<0.
\]
\end{proof}

\textbf{Proof of Theorem \ref{Thm:normalized-Subcritical-solutions-infinity}.} Fix \(\mu>0\) and \(0<c<c^*(\mu)\). By Lemma \ref{lem:Sigma-neq-emptyset}, we have \(a_m<0\) for every \(m\in\mathbb N^+\). Applying Proposition \ref{Pro:minimax-theorem} with \(X=H_r^2(\mathbb R^N)\), \(H=L^2(\mathbb R^N)\) and \(k=c\), so that \(Q=S_r(c)\), we obtain that critical points of \(J|_{S_r(c)}\) at the minimax levels \(a_m\). Since every negative sublevel of \(J\) is contained in \(V_r(c)\) and \(J=I\) on \(V_r(c)\), these critical points are also constrained critical points of \(I|_{V_r(c)}\). As in the proof of Theorem \ref{Thm:normalized-Subcritical-solutions-existence}, symmetric criticality implies that these points are also constrained critical points of \(I\) on \(S(c) \) and hence normalized weak solutions of (\ref{eqn:BS-equation-L2-Super+Critical}). If several consecutive minimax levels coincide, Proposition \ref{Pro:minimax-theorem}-\((ii)\) guarantees the corresponding genus multiplicity. Consequently, \(I|_{V_r(c)}\) admits at least \(m\) distinct pairs of constrained critical points for every \(m\in\mathbb N^+\). Moreover, Lemma \ref{Lem:local-ps-satisfies} ensures that their corresponding Lagrange multipliers are negative. Finally, Proposition \ref{Pro:minimax-theorem}-\((iii)\) yields \(a_m\rightarrow 0^-\) as \(m\to+\infty\). Choosing a critical point \(u_m\) at the level \(a_m\), we obtain that \(I(u_m)=J(u_m)=a_m\rightarrow 0^-\).
\\
\qed

\section{Proof of Theorem \ref{Thm:normalized-bScritical-solutions-L2super}}\label{sec:proof-main-theorem-3-and-4}

In this section, we construct a bounded radial constrained Palais-Smale sequence at the mountain pass level and establish the compactness properties needed for Theorem \ref{Thm:normalized-bScritical-solutions-L2super}.
\begin{lemma}\label{Lem:Delta-mathcalH-property}
Let $\bar{p}< p<4^*$ and $c,\mu>0$. For any fixed $u\in S_r(c)$, one has $I(\mathcal H(u,s))\to0^+$ as $s\to -\infty$ and $I(\mathcal H(u,s))\to-\infty$ as $s\to +\infty$.
\end{lemma}
\begin{proof}
Direct calculations give
$$
\|\Delta\mathcal{H}(u,s)\|^2_2=e^{4s}\|\Delta u\|^2_2,\ \ \|\nabla\mathcal{H}(u,s)\|^2_2=e^{2s}\|\nabla u\|^2_2,\ \ \|\mathcal{H}(u,s)\|^p_p=e^{2p\gamma_p s}\|u\|^p_p
$$
and
$$
\|\mathcal{H}(u,s)\|^{4^*}_{4^*}=e^{24^* s}\|u\|^{4^*}_{4^*}.
$$
Since $\bar{p}<p<4^*$ implies that $p\gamma_p\geq2$, it is obvious to see that
\begin{equation}\label{eqn:I-to-0-infty}
\begin{aligned}
I(\mathcal H(u,s))&=\frac{e^{4s}}{2}\int_{\mathbb{R}^N}|\Delta u|^2dx+\frac{e^{2s}}{2}\int_{\mathbb{R}^N}|\nabla u|^2dx-\frac{\mu e^{2p\gamma_p s}}{p}\int_{\mathbb{R}^N}|u|^p dx-\frac{e^{24^* s}}{4^*}\int_{\mathbb{R}^N}|u|^{4^*}dx\\
&\ \ \ \begin{cases}\to 0^+, & s\to -\infty; \\
\to -\infty, & s\to +\infty. \\
\end{cases}
\end{aligned}
\end{equation}
\end{proof}

\begin{lemma}\label{Prop:I-u-maximum}
Let $\bar{p}< p<4^*$ and $c,\mu>0$. For any $u\in S_r(c)$, we have:
\begin{itemize}
\item[(i)] there exists a unique $s_u\in\mathbb R$ such that $\mathcal H(u,s_u)\in \mathcal P_r(c)$ and $s_u$ is a strict maximum point for the functional $I(\mathcal H(u,s))$ at a positive level. Moreover, $s_u=0$ when $u\in \mathcal P_r(c)$;
\item[(ii)] if $I(u)\leq 0$, then $P(u)<0$.
\end{itemize}
\end{lemma}
\begin{proof}
$(i)$ In view of Lemma \ref{Lem:Delta-mathcalH-property}, $I(\mathcal H(u,s))$ has a global maximum point $s_u\in\mathbb R$ at a positive level.
Thus, we derive that
\begin{equation}\label{eqn:partialI-0-p}
\begin{aligned}
0&=\frac{\partial}{\partial s}I(\mathcal H(u,s_u))=2e^{4s_u}\|\Delta u\|^2_2+e^{2s_u}\|\nabla u\|^2_2-2\mu\gamma_pe^{2p\gamma_p s_u}\|u\|^p_p -2e^{24^*s_u}\|u\|^{4^*}_{4^*}=2P(\mathcal H(u,s_u)).
\end{aligned}
\end{equation}
That is to say, $\mathcal H(u,s_u)\in \mathcal{P}_r(c)$. To establish uniqueness, divide \(P(H(u,s))=0\) by \(e^{4s}>0\). This equation is equivalent to
\[
\begin{aligned}
F_u(s):=\|\Delta u\|_2^2+\frac12e^{-2s}\|\nabla u\|_2^2-\mu\gamma_p e^{2(p\gamma_p-2)s}\|u\|_p^p-e^{2(4^*-2)s}\|u\|_{4^*}^{4^*}=0.
\end{aligned}
\]
Since \(\bar p<p<4^*\) implies \(p\gamma_p>2\) and \(4^*>2\), differentiation gives
\[
\begin{aligned}
F_u'(s)=-e^{-2s}\|\nabla u\|_2^2-2(p\gamma_p-2)\mu\gamma_pe^{2(p\gamma_p-2)s}\|u\|_p^p-2(4^*-2)e^{2(4^*-2)s}\|u\|_{4^*}^{4^*}<0.
\end{aligned}
\]
Moreover,
\[
\lim_{s\to-\infty}F_u(s)=+\infty,
\qquad
\lim_{s\to+\infty}F_u(s)=-\infty.
\]
Therefore, \(F_u\) has exactly one zero. Since \(\frac{\partial}{\partial s}I(H(u,s))=2P(H(u,s))=2e^{4s}F_u(s)\), the functional \(I(H(u,s))\) is strictly increasing before this zero and strictly decreasing afterwards. Hence, \(s_u\) is the unique strict maximum point. For any \(u\in \mathcal P(c)\), differentiating \(P(H(u,s))\) at \(s=0\) and using \(P(u)=0\), we obtain
\[
\begin{aligned}
\left.\frac{d}{ds}P(H(u,s))\right|_{s=0}
&=-\|\nabla u\|_2^2-2(p\gamma_p-2)\mu\gamma_p\|u\|_p^p-2(4^*-2)\|u\|_{4^*}^{4^*}<0.
\end{aligned}
\]
This nonvanishing derivative establishes the required nondegeneracy of the Poho\v zaev constraint in the mass preserving scaling direction.

$(ii)$ By $(i)$, the fiber map \(s\mapsto I(H(u,s))\) is positive on \((-\infty,s_u]\) and strictly decreasing on \([s_u,+\infty)\). If \(I(u)\leq0\), then \(s_u\geq0\) is impossible because the fiber map is positive on \((-\infty,s_u]\). Hence, \(s_u<0\). Consequently, the fiber map is strictly decreasing at \(s=0\), and therefore \(2P(u)=\frac{\partial}{\partial s}I(H(u,s))\Big|_{s=0}<0\). This completes the proof.
\end{proof}

\begin{lemma}\label{Lem:supAI-infBI}
Let $\bar{p}<p<4^*$ and $c,\mu>0$. Define
\begin{equation}\label{eqn:K-define}
\begin{aligned}
K(c,\mu)&:=\!\min\!\Big\{\frac{8}{4^*}\Big(\frac{4^*S}{16}\Big)^\frac{N}{4}, 4\Big(\frac{S}{4}\Big)^\frac{N}{4},\Big(\frac{p}{8C^p_{N,p}2^\frac{p\gamma_p}{2}\mu c^\frac{p\!-\!p\gamma_p}{2}}\Big)^\frac{2}{p\gamma_p\!-\!2}, \Big(\frac{p}{(p\!-\!2)NC^p_{N,p}\mu c^\frac{p\!-\!p\gamma_p}{2}}\Big)^\frac{2}{p\gamma_p\!-\!2}\Big\},
\end{aligned}
\end{equation}
\begin{equation}\label{eqn:Defn-mathcal-A}
\mathcal A:=\{u\in S_r(c):\|\Delta u\|^2_2+\|\nabla u\|^2_2\leq K(c,\mu)\}
\end{equation}
and
\begin{equation}\label{eqn:Defn-mathcal-B}
\mathcal B:=\{u\in S_r(c):\|\Delta u\|^2_2+\|\nabla u\|^2_2=2K(c,\mu)\}.
\end{equation}
For given $c>0$, we see that:
\begin{itemize}
\item[(i)] $0<\sup\limits_{u\in \mathcal A}I(u)<\inf\limits_{u\in \mathcal B}I(u)$;
\item[(ii)] there holds that $I(u)>0$ for any $u\in S_r(c)$ with $\|\Delta u\|^2_2+\|\nabla u\|^2_2\leq K(c,\mu)$. Moreover,
$$
I_*:=\inf \left\{I(u):u\in S_r(c)\ and \ \|\Delta u\|^2_2+\|\nabla u\|^2_2=\frac{K(c,\mu)}{2}\right\}>0.
$$
\end{itemize}
\end{lemma}
\begin{proof}
$(i)$ For the case $\bar{p}<p<4^*$, if $u\in \mathcal A$, $v\in \mathcal B$, we derive from (\ref{eqn:GNinequality}), (\ref{eqn:Defn-S}) and (\ref{eqn:K-define}) that
\begin{equation}\label{eqn:Iv-minus-Iu-AB}
\begin{aligned}
I(v)-I(u)
&\geq\frac{1}{4}K(c,\mu)\!-\frac{\mu}{p}C^p_{N,p}2^\frac{p\gamma_p}{2}c^\frac{p-p\gamma_p}{2}\Big(K(c,\mu)\Big)^\frac{p\gamma_p}{2}+\frac{1}{4}K(c,\mu)\!
-\!\frac{2^\frac{4^*}{2}}{4^* S^\frac{4^*}{2}}\!\Big(K(c,\mu)\Big)^\frac{4^*}{2}
> 0.
\end{aligned}
\end{equation}
By the definition of \(K(c,\mu)\), the two nonlinear terms in the energy functional are each bounded above by \(\frac{K(c,\mu)}{8}\) on \(\mathcal B\). Consequently, \(\sup_{u\in \mathcal A}I(u)\leq\frac{K(c,\mu)}2\) and \(\inf_{u\in \mathcal B}I(u)\geq\frac{3K(c,\mu)}4\). Therefore, the two sets are separated by a uniform positive energy gap of at least \(\frac{K(c,\mu)}{4}\).

$(ii)$ From (\ref{eqn:GNinequality})-(\ref{eqn:Defn-S}), we have \(I(u)\geq\|\Delta u\|^2_2f(c,\|\Delta u\|^2_2)\), where
$$
f(c,r):=\frac{1}{2}-\frac{\mu}{p}C^p_{N,p}c^\frac{p-p\gamma_p}{2}r^\frac{p\gamma_p-2}{2}-\frac{1}{4^*S^\frac{4^*}{2}}r^\frac{4^*-2}{2},\ \ \ r\geq 0.
$$
For any $u\in S_r(c)$ with $\|\Delta u\|^2_2\leq\|\Delta u\|^2_2+\|\nabla u\|^2_2\leq K(c,\mu)$, observe that
$$
\begin{aligned}
f(c,\|\Delta u\|^2_2)&\geq\frac{1}{2}\!-\frac{\mu}{p}C^p_{N,p}2^\frac{p\gamma_p}{2}c^\frac{p-p\gamma_p}{2}\Big(K(c,\mu)\Big)^\frac{p\gamma_p-2}{2}\!
-\!\frac{2^\frac{4^*}{2}}{4^* S^\frac{4^*}{2}}\!\Big(K(c,\mu)\Big)^\frac{4^*-2}{2}
>0.
\end{aligned}
$$
As a consequence, $I(u)>0$ for any $ u\in S_r(c)$ with $\|\Delta u\|^2_2+\|\nabla u\|^2_2\leq K(c,\mu)$.

If $\|\Delta u\|^2_2+\|\nabla u\|^2_2=\frac{K(c,\mu)}{2}$, according to the definition of $K(c,\mu)$, we readily infer that
$$
\begin{aligned}
I(u)
&> \frac{K(c,\mu)}{2}\Big[\frac{1}{4}-\frac{\mu}{p}C^p_{N,p}c^\frac{p-p\gamma_p}{2}\Big(K(c,\mu)\Big)^\frac{p\gamma_p-2}{2}+\frac{1}{4}
-\frac{1}{4^*S^\frac{4^*}{2}}\Big(K(c,\mu)\Big)^\frac{4^*-2}{2}\Big]
>0,
\end{aligned}
$$
which states that $I_*>0$. The same estimates also provide a positive lower bound for \(I(u)\) on the sphere \(\| \Delta u\|_2^2+\|\nabla u\|_2^2=\frac{K(c,\mu)}2\).
\end{proof}

\begin{remark}\label{rem:MP-level}
{\rm
For any given ${u}\in S_r(c)$, Lemmas \ref{Lem:Delta-mathcalH-property} and \ref{Lem:supAI-infBI} yield two numbers
\begin{equation}\label{eqn:Defn-s1-s2}
s_1:=s_1(c,\mu,{u})<0 \quad {\rm and} \quad  s_2:=s_2(c,\mu,{u})>0
\end{equation}
to guarantee that the functions
\begin{equation}\label{eqn:Defn-hat-u1-u2}
\hat{u}_{1,\mu}:=\mathcal H({u},s_1) \quad {\rm and} \quad \hat{u}_{2,\mu}:=\mathcal H({u},s_2)
\end{equation}
satisfy
$$
\|\Delta \hat{u}_{1,\mu}\|^2_2+\|\nabla \hat{u}_{1,\mu}\|^2_2<\frac{K(c,\mu)}{2},\ \ \|\Delta \hat{u}_{2,\mu}\|^2_2+\|\nabla \hat{u}_{2,\mu}\|^2_2>2K(c,\mu),\ \ I(\hat{u}_{1,\mu})>0\ \ \mbox{and}\ \ I(\hat{u}_{2,\mu})<0.
$$
Therefore, following the idea from Jeanjean \cite{Jeanjean1997}, we can determine the following mountain pass level denoted by
\begin{equation}\label{eqn:Defn-gamma-mu}
\gamma_\mu(c):=\inf\limits_{h\in\Gamma(c)}\max\limits_{t\in[0,1]}I(h(t))
\end{equation}
with
\begin{equation}\label{eqn:Defn-Gamma-c}
\Gamma(c):=\Big\{h\in C([0,1],S_r(c)):\|\Delta h(0)\|^2_2+\|\nabla h(0)\|^2_2<\frac{K(c,\mu)}{2}\ \mbox{and}\ I(h(1))<0\Big\}.
\end{equation}
To guarantee the existence of a (PS) sequence for \(I\) at the level \(\gamma_\mu(c)\) defined in (\ref{eqn:Defn-gamma-mu}), we introduce the following auxiliary functional \(\widetilde I:E=H^2(\mathbb R^N)\times\mathbb R\rightarrow\mathbb R\):
\begin{equation}\label{eqn:Defn-tilde-I}
\tilde{I}(u,s):=\frac{e^{4s}}{2}\|\Delta u\|^2_2+\frac{e^{2s}}{2}\|\nabla u\|^2_2-\frac{\mu e^{2p\gamma_p s}}{p}\|u\|^p_p-\frac{e^{24^* s}}{4^*}\|u\|^{4^*}_{4^*}
=I(\mathcal H(u,s)).
\end{equation}
In the radial minimax argument, we use its restriction to \(S_r(c)\times\mathbb R\subset E_r\). As the functional $\tilde{I}$ is concerned, it is necessary to consider the minimax level
\begin{equation}\label{eqn:Defn-widetilde-gamma}
\widetilde{\gamma}_\mu(c):=\inf\limits_{\tilde{h}\in\widetilde{\Gamma}(c)}\max\limits_{t\in[0,1]}\tilde{I}(\tilde{h}(t)),
\end{equation}
where the path family $\tilde{\Gamma}(c)$ is supposed to satisfy the following requirement
\begin{equation}\label{eqn:Defn-widetilde-Gamma}
\begin{aligned}
\widetilde{\Gamma}(c):=\Big\{\tilde{h}=&(\tilde{h}_1,\tilde{h}_2)\in C([0,1],S_r(c)\times \mathbb{R}):\tilde{h}(0)=(\tilde{h}_{1}(0),0),\ \ \tilde{h}(1)=(\tilde{h}_{1}(1),0),\\
&\ \ \ \ \ \ \ \ \ \ \ \ \ \ \ \ \ \ \ \ \ \ \ \ \ \ \|\Delta \tilde{h}_{1}(0)\|^2_2+\|\nabla \tilde{h}_{1}(0)\|^2_2<\frac{K(c,\mu)}{2} \ \ \mbox{and}\ \ \tilde{I}(\tilde{h}_{1}(1),0)<0\Big\}.
\end{aligned}
\end{equation}
}
\end{remark}

\begin{lemma}\label{Lem:gamma-tildegamma-equal}
Let $\bar{p}< p<4^*$ and $c,\mu>0$. Then, we have
\begin{equation}\label{eqn:gamma-mu-widegamma-mu}
\widetilde{\gamma}_\mu(c)=\gamma_\mu(c)=m_r(c):=\inf\limits_{u\in \mathcal P_r(c)}I(u) >0.
\end{equation}
\end{lemma}

\begin{proof}
Firstly, we prove $\widetilde{\gamma}_\mu(c)=\gamma_\mu(c)$. For any $\tilde{h}(t)\in\widetilde{\Gamma}(c)$, it can be rewritten as $\tilde{h}(t)=(\tilde{h}_1(t),\tilde{h}_2(t))\in S_r(c)\times \mathbb R$. Set
$h(t)=\mathcal H (\tilde{h}_1(t),\tilde{h}_2(t))$, then $h(t)\in \Gamma(c)$ and
$$
\max\limits_{t\in[0,1]}\tilde{I}(\tilde{h}(t))=\max\limits_{t\in[0,1]}I(h(t))\geq \gamma_\mu(c),
$$
which implies that $\widetilde{\gamma}_\mu(c)\geq\gamma_\mu(c)$. On the other hand, for any $h\in\Gamma(c)$, denote $\tilde{h}(t)=(h(t),0)$, then $\tilde{h}(t)\in\widetilde{\Gamma}(c)$ and
$$
\widetilde{\gamma}_\mu(c)\leq\max\limits_{t\in[0,1]}\tilde{I}(\tilde{h}(t))
=\max\limits_{t\in[0,1]}I(h(t)),
$$
it means that $\gamma_\mu(c)\geq\widetilde{\gamma}_\mu(c)$. Therefore, we conclude that $\widetilde{\gamma}_\mu(c)=\gamma_\mu(c)$.

Secondly, we state that $\gamma_\mu(c)=m_r(c)$. For any $u\in\mathcal P_r(c)$, Lemmas \ref{Lem:Delta-mathcalH-property} and \ref{Lem:supAI-infBI} imply the existence of two constants $s_1=s_1(c,\mu,u)<0$ and
$s_2=s_2(c,\mu,u)>0$ to guarantee that $h(t):=\mathcal H(u,(1-t)s_1+ts_2)$, $t\in[0,1]$, belongs to $\Gamma(c)$.
Therefore, it follows from Lemma \ref{Prop:I-u-maximum} that \(\gamma_\mu(c)\leq\max\limits_{t\in[0,1]}I(h(t))=\max\limits_{t\in[0,1]}I(\mathcal H(u,(1-t)s_1+ts_2))\leq I(u)\), namely, $\gamma_\mu(c)\leq m_r(c)$.

On the other hand, for any \(\widetilde h\in\widetilde\Gamma(c)\), write \(\widetilde h(t)=(\widetilde h_1(t),\widetilde h_2(t))\) and define \[\widetilde P(t):=P\big(H(\widetilde h_1(t),\widetilde h_2(t))\big).\]
From (\ref{eqn:Pohozaev-identify}), (\ref{eqn:GNinequality}) and (\ref{eqn:Defn-S}), we have
\begin{equation}\label{eqn:P-leq-g}
\begin{aligned}
P(u)\geq\|\Delta u\|^2_2 g(c,\|\Delta u\|^2_2),
\end{aligned}
\end{equation}
where \(g(c,r):=1-\mu\gamma_p C_{N,p}^p c^\frac{p-p\gamma_p}{2}r^\frac{p\gamma_p-2}{2}-\frac{1}{S^\frac{4^*}{2}}r^\frac{4^*-2}{2}\) and \(r\geq 0\).
For any $u\in \mathcal A$ defined in (\ref{eqn:Defn-mathcal-A}), from the definition of $K(c,\mu)>0$ in (\ref{eqn:K-define}), we see that
$$
\begin{aligned}
g(c,\|\Delta u\|^2_2)\geq\frac{1}{2}+\frac{1}{4}-\mu\gamma_p C_{N,p}^p c^\frac{p-p\gamma_p}{2}(K(c,\mu))^\frac{p\gamma_p-2}{2}+\frac{1}{4}-\frac{1}{S^\frac{4^*}{2}}\Big(K(c,\mu)\Big)^\frac{4^*-2}{2}>0,
\end{aligned}
$$
which brings that
\begin{equation}\label{eqn:P-geq-0}
P(u)>0, \ \ \  u\in \mathcal A.
\end{equation}
Thereby, utilizing (\ref{eqn:P-geq-0}) and Lemma \ref{Prop:I-u-maximum} $(ii)$, we infer that
$$
\widetilde{P}(0)=P\left(\mathcal H(\tilde{h}_1(0), \tilde{h}_2(0))\right)=P\left(\tilde{h}_1(0)\right)>0
$$
and
$$
\widetilde{P}(1)=P\left(\mathcal H(\tilde{h}_1(1), \tilde{h}_2(1))\right)=P\left(\tilde{h}_1(1)\right)<0.
$$
By the continuity of \(\widetilde P(t)\), there exists some \(\tilde t\in(0,1)\) such that
\(\widetilde P(\tilde t)=0\), which means that $\mathcal H(\tilde{h}_1(\tilde{t}), \tilde{h}_2(\tilde{t})) \in \mathcal P_r(c)$. Consequently, one has \(\max_{t \in[0,1]} \tilde{I}(\tilde{h}(t))=\max_{t \in[0,1]} I\left(\mathcal H(\tilde{h}_1(t), \tilde{h}_2(t))\right) \geq \inf_{u \in \mathcal P_r(c)}I(u)\).
Due to the arbitrariness of $\tilde{h}(t)$, it signifies that $\widetilde{\gamma}_\mu(c)=\gamma_\mu(c) \geq m_r(c)$. Hence, $\gamma_\mu(c) = m_r(c)$.

Finally, we claim that $m_r(c)>0$. For any $u\in \mathcal P_r(c)$, by using (\ref{eqn:GNinequality}) and (\ref{eqn:Defn-S}), we get
$$
\begin{aligned}
\frac{1}{2}\Big(\|\Delta u\|^2_2+\|\nabla u\|^2_2\Big)\leq \mu\gamma_p C_{N,p}^p c^\frac{p-p\gamma_p}{2}\Big(\|\Delta u\|^2_2+\|\nabla u\|^2_2\Big)^\frac{p\gamma_p}{2}+\frac{1}{S^\frac{4^*}{2}}\Big(\|\Delta u\|^2_2+\|\nabla u\|^2_2\Big)^\frac{4^*}{2}.
\end{aligned}
$$
Together with $\frac{4^*}{2}>1$ and $\frac{p\gamma_p}{2}>1$, this implies that there exists $\xi>0$ such that 
\[
\inf_{u\in \mathcal P_r(c)} \Big(\|\Delta u\|^2_2+\|\nabla u\|^2_2\Big) \geq \xi.
\]
Therefore, for every $u\in \mathcal P_r(c)$, we have
\begin{equation}\label{eqn:Iu-geq-strict-xi-0}
\begin{aligned}
I(u)=I(u)-\frac{1}{4^*}P(u)\geq\Big(\frac{1}{2}-\frac{1}{p\gamma_p}\Big)\Big(\|\Delta u\|^2_2+\|\nabla u\|^2_2\Big)\geq\Big(\frac{1}{2}-\frac{1}{p\gamma_p}\Big)\xi>0,
\end{aligned}
\end{equation}
since $p\gamma_p-4^*=\frac{N[(N-4)p-2N]}{4(N-4)}<0$ and \(\mu\gamma_p\|u\|^p_p=\|\Delta u\|^2_2+\frac{1}{2}\|\nabla u\|^2_2-\|u\|^{4^*}_{4^*}\leq \|\Delta u\|^2_2+\frac{1}{2}\|\nabla u\|^2_2\).
\end{proof}
Let \(K=K(c,\mu)\). By Lemma  \ref{Lem:supAI-infBI}, every admissible path crosses \(\mathcal B\). Hence, \(\widetilde\gamma_\mu(c)=\gamma_\mu(c)\geq\frac{3K}{4}\).
Moreover, for every \(\widetilde h\in\widetilde\Gamma(c)\), its initial endpoint satisfies
\(I(\widetilde h_1(0))\leq\frac K4\), whereas its final endpoint has negative energy. Consequently,
\begin{equation}\label{eqn:gamma-mu-c-big-gamma-0-mu-c}
	\begin{aligned}
		\widetilde{\gamma}_\mu(c)&=\gamma_\mu(c)>\frac{K}{4}\geq\sup_{\tilde{h}\in \widetilde{\Gamma}(c)}\{I(h_1(0)),I(h_1(1))\}=\sup_{\tilde{h}\in \widetilde{\Gamma}(c)}\{\tilde{I}(\tilde{h}(0)),\tilde{I}(\tilde{h}(1))\}=:(\widetilde{\gamma}_\mu)_0(c).
	\end{aligned}
\end{equation}

\begin{lemma}\label{Lem:three-tildeJ-prove}
Let $\bar{p}< p<4^*$ and $c,\mu>0$, for $0<\varepsilon<\widetilde{\gamma}_\mu(c)-(\widetilde{\gamma}_\mu)_0(c)$, choose $g\in \widetilde{\Gamma}(c)$ such that $\max\limits_{t\in[0,1]}\tilde{I}(g(t)) \leq \widetilde{\gamma}_\mu(c)+\varepsilon$. Then there exists $(v,s)\in S_r(c)\times \mathbb R$ to guarantee that
\begin{itemize}
\item[(i)]  $\tilde{I}(v,s)\in\left[\widetilde{\gamma}_\mu(c)-\varepsilon,\widetilde{\gamma}_\mu(c)+\varepsilon\right]$;
\item[(ii)] $\min\limits_{t\in[0,1]}\|(v,s)-g(t)\|_E\leq \sqrt{\varepsilon}$;
\item[(iii)]$\|(\tilde{I}|_{S_r(c)\times\mathbb R})'(v,s)\|_{E'}\leq 2\sqrt{\varepsilon}$,
\end{itemize}
where $\widetilde{\gamma}_\mu(c)$, $\widetilde{\Gamma}(c)$ and $(\widetilde{\gamma}_\mu)_0(c)$are given in (\ref{eqn:Defn-widetilde-gamma}), (\ref{eqn:Defn-widetilde-Gamma}) and (\ref{eqn:gamma-mu-c-big-gamma-0-mu-c}), respectively.
\end{lemma}
\begin{proof} 
The proof follows the argument of \cite[Lemma 2.5]{Liu-Zhang2023}, so we omit the details.
\end{proof}

The following proposition can be obtained directly from the above lemma.

\begin{proposition}\label{Pro:three-tildeJ}
Let $\bar{p}< p<4^*$ and $c,\mu>0$. Choose $\{g_n\}\subset\widetilde{\Gamma}(c)$ such that $\max\limits_{t\in[0,1]}\tilde{I}(g_n(t))\leq \widetilde{\gamma}_\mu(c)+\frac{1}{n}$, then, for all sufficiently large $n$, there exist $(v_n,s_n)\in S_r(c)\times \mathbb{R} $ such that
\begin{itemize}
\item[(i)]  $\tilde{I}(v_n,s_n)\in\left[\widetilde{\gamma}_\mu(c)-\frac{1}{n},\widetilde{\gamma}_\mu(c)+\frac{1}{n}\right]$;
\item[(ii)] $\min\limits_{t\in[0,1]}\|(v_n,s_n)-g_n(t)\|_E\leq \frac{1}{\sqrt{n}}$;
\item[(iii)]$\|(\tilde{I}|_{S_r(c)\times\mathbb R})'(v_n,s_n)\|_{E'}\leq \frac{2}{\sqrt{n}}$, that is,
$$
|\langle \tilde{I}'(v_n,s_n),\omega \rangle_{E'\times E}|
\leq \frac{2}{\sqrt{n}}\|\omega\|_E
$$ for all $\omega \in \widetilde{T}_{(v_n,s_n)}:=\{(\omega_1,\omega_2)\in E_r,\langle v_n,\omega_1 \rangle_{L^2}=0\}$.
\end{itemize}
\end{proposition}

\begin{lemma}\label{Lem:PS-true}
Let \(\bar{p} < p < 4^*\) and \(c, \mu > 0\). In Proposition \ref{Pro:three-tildeJ}, choose the paths \(g_n(t) = ((g_n)_1(t), 0)\) for \(t \in [0, 1]\), which is possible by Lemma \ref{Lem:gamma-tildegamma-equal}. For the corresponding sequence \(\{(v_n, s_n)\}\) supplied by Proposition \ref{Pro:three-tildeJ}, set \(u_n:=\mathcal H(v_n,s_n)=e^{\frac{Ns_n}{2}}v_n(e^{s_n} x)\). Then the following properties hold:
\begin{itemize}
\item[(i)]  $I(u_n)\to\gamma_\mu(c)$ as $n\to\infty$;
\item[(ii)] $\partial_s \tilde{I}(v_n,s_n)=2\|\Delta u_n\|^2_2+\|\nabla u_n\|^2_2-2\mu \gamma_p\|u_n\|^p_p-2\|u_n\|^{4^*}_{4^*}\to 0$
as $n\to\infty$;
\item[(iii)] $\{\|u_n\|_{H^2}\}$ and $\{\frac{\mu}{p}\|u_n\|^p_p+\frac{1}{4^*}\|u_n\|^{4^*}_{4^*}\}$ are bounded;
\item[(iv)] $|\langle I'(u_n),\omega \rangle_{(H^2(\mathbb{R}^N))'\times H^2(\mathbb{R}^N)}|\leq \frac{2e^4}{\sqrt{n}}\|\omega\|_{H^2}$ for all
$\omega\in T_{u_n}\!\!:=\{\omega\in H^2(\mathbb R^N),\langle u_n,\omega \rangle_{L^2}=0\}$.
\end{itemize}
\end{lemma}
\begin{proof}
Observe that $I(u_n)=\tilde{I}(v_n,s_n)$ and $\widetilde{\gamma}_\mu(c)=\gamma_\mu(c)$, \emph{(i)} is a direct conclusion from Proposition \ref{Pro:three-tildeJ}-\emph{(i)}.

\emph{(ii)} Letting $\partial_s \tilde{I}(v_n,s_n)=\langle \tilde{I}'(v_n,s_n),(0,1)\rangle_{E'\times E}$, in view of Proposition \ref{Pro:three-tildeJ}-\emph{(iii)}, we see that \(\partial_s \tilde{I}(v_n,s_n)\to 0\) as \(n\to \infty\). More precisely, a simple calculation indicates that
$$
\partial_s \tilde{I}(v_n,s_n)=2\|\Delta u_n\|^2_2+\|\nabla u_n\|^2_2-2\mu \gamma_p\|u_n\|^p_p-2\|u_n\|^{4^*}_{4^*}\to 0\ \mbox{as}\ n\to \infty.
$$

\emph{(iii)} Since $\{\tilde{I}(v_n,s_n)\}$ is bounded, we have
\begin{equation}\label{eqn:inrquality-2}
-C_1\leq\tilde{I}(v_n,s_n)=I(u_n)=\frac{1}{2}\Big(\|\Delta u_n\|^2_2+\|\nabla u_n\|^2_2\Big)-\Big(\frac{\mu}{p}\|u_n\|^p_p+\frac{1}{4^*}\|u_n\|^{4^*}_{4^*}\Big)\leq C_1\\
\end{equation}
for some constant $C_1>0$. When $\bar{p}<p<4^*$, the boundedness of $\{\tilde{I}(v_n,s_n)\}$ and $\{\partial_s \tilde{I}(v_n,s_n)\}$ also brings that for some constant $C_2>0$,
 \begin{equation}\label{eqn:inrquality-1}
\begin{aligned}
-C_2&\leq N\tilde{I}(v_n,s_n)+\partial_s \tilde{I}(v_n,s_n)
\leq\frac{N+4}{2}\Big(\|\Delta u_n\|^2_2+\|\nabla u_n\|^2_2\Big)-\frac{Np}{2}\Big(\frac{\mu}{p}\|u_n\|^p_p+\frac{1}{4^*}\|u_n\|^{4^*}_{4^*}\Big),
\end{aligned}
\end{equation}
since $\frac{N}{4^*}+2-\frac{Np}{24^*}=\frac{2N-(N-4)p}{4}>0$. Combining (\ref{eqn:inrquality-2}) and (\ref{eqn:inrquality-1}), we deduce that
$$
\begin{aligned}
-C_2\leq (N+4)C_1-\Big(\frac{Np}{2}-N-4\Big)\Big(\frac{\mu}{p}\|u_n\|^p_p+\frac{1}{4^*}\|u_n\|^{4^*}_{4^*}\Big).
\end{aligned}
$$
Since \(2+\frac{8}{N}<p<4^*\), the above inequality yields
\begin{equation}\label{varrho-1-define}
\frac{\mu}{p}\|u_n\|_p^p+\frac{1}{4^*}\|u_n\|_{4^*}^{4^*}\leq\frac{(N+4)C_1+C_2}{\frac{Np}{2}-N-4}=:\varrho_1.
\end{equation}
Consequently,
\begin{equation}\label{varrho-2-define}
\begin{aligned}
\|\Delta u_n\|_2^2+\|\nabla u_n\|_2^2\leq 2\left(\frac{\mu}{p}\|u_n\|_p^p+\frac{1}{4^*}\|u_n\|_{4^*}^{4^*}\right)+2C_1
\leq 2\varrho_1+2C_1=:\varrho_2.
\end{aligned}
\end{equation}

\emph{(iv)} For $h_n\in T'_{u_n}:=\{\omega\in H^2_r(\mathbb R^N),\langle u_n,\omega \rangle_{L^2}=0\}$ and $\tilde{h}_n(x):=e^{-\frac{Ns_n}{2}}h_n(e^{-s_n}x)$, we have
\(\langle I'(u_n),h_n \rangle_{(H^2)'\times H^2}=\langle \tilde{I}'(v_n,s_n),(\tilde{h}_n,0) \rangle_{E'\times E}\). In addition, since $h_n\in T'_{u_n}$, that is, $\langle u_n,h_n \rangle_{L^2}=0$, it gives that \(0=\int_{\mathbb R^N}e^\frac{Ns_n}{2}v_n(e^{s_n}x)h_n(x)dx=\int_{\mathbb R^N}v_n(x)e^{-\frac{Ns_n}{2}}h_n(e^{-s_n}x)dx=\langle v_n, \tilde{h}_n\rangle_{L^2}\).
Hence, $(\tilde{h}_n,0)\in \widetilde{T}_{(v_n,s_n)}$, where $\widetilde{T}_{(v_n,s_n)}$ is given in Proposition \ref{Pro:three-tildeJ}. Using this fact and Proposition
\ref{Pro:three-tildeJ}-\emph{(iii)}, we see that \(|\langle I'(u_n),h_n \rangle_{(H^2)'\times H^2}|=|\langle \tilde{I}'(v_n,s_n),(\tilde{h}_n,0) \rangle_{E'\times E}|\leq \frac{2}{\sqrt{n}}\|(\tilde{h}_n,0)\|_{E}\).
Therefore, it only remains to verify that $\|(\tilde{h}_n,0)\|_{E}\leq e^4\|h_n\|_{H^2}$ for all $n$. Since the paths $g_n$ were chosen with zero second component, Proposition \ref{Pro:three-tildeJ}-\emph{(ii)} gives
 \(|s_n|^2=|s_n-0|^2\leq\min\limits_{t\in[0,1]}\|(v_n,s_n)-g_n(t)\|^2_E\leq\frac{1}{n}\).
Consequently, we obtain that
$$
\begin{aligned}
\|(\tilde{h}_n,0)\|^2_{E}
\leq e^4\Big(\int_{\mathbb R^N}|\Delta h_n(x)|^2dx+\int_{\mathbb R^N}|\nabla h_n(x)|^2dx+\int_{\mathbb R^N}|h_n(x)|^2dx\Big)
= e^4\|h_n\|^2_{H^2}.
\end{aligned}
$$
The estimate first holds for radial tangent vectors. To extend it to the full tangent space, let \(\omega\in H^2(\mathbb R^N)\) satisfy \(\langle u_n,\omega\rangle_{L^2}=0\). Define its rotational average by \(\omega^\#(x)=\int_{O(N)}\omega(Rx)\,dR\), where \(dR\) denotes the normalized Haar measure on \(O(N)\). Since \(u_n\) is radially symmetric and \(I\) is rotationally invariant, we have \(\langle u_n,\omega^\#\rangle_{L^2}=0\), \(\|\omega^\#\|_{H^2}\leq\|\omega\|_{H^2}\) and 
\(\langle I'(u_n),\omega\rangle=\langle I'(u_n),\omega^\#\rangle\). Therefore, the estimate extends to the full tangent space.
\end{proof}
We next study the compactness of the sequence \(\{u_n\}\) from Lemma \ref{Lem:PS-true} when \(0<m_r(c)<\frac{2}{N}S^\frac{N}{4}\)
\begin{lemma}\label{pro:solution-either-or-true}
Let \(N\geq5\), \(\bar p<p<4^*\) and \(c,\mu>0\). Suppose that \(\{u_n\}\subset S_r(c)\) is the sequence given in Lemma \ref{Lem:PS-true} and \(0<m_r(c)<\frac{2}{N}S^\frac{N}{4}\).
Then, after passing to a subsequence, there exist \(\bar u\in H_r^2(\mathbb R^N)\setminus\{0\}\) and \(\bar\lambda\in\mathbb R\) such that
\[
u_n\rightharpoonup\bar u\quad\text{in }H_r^2(\mathbb R^N),
\qquad
u_n\to\bar u\quad\text{in }L^p(\mathbb R^N),
\qquad
\lambda_n\to\bar\lambda.
\]
Moreover,
\[
\begin{gathered}
\Delta u_n\to\Delta\bar u,\quad\nabla u_n\to\nabla\bar u
\quad\text{in}\quad L^2(\mathbb R^N),\quad u_n\to\bar u\quad\text{in} \quad L^{4^*}(\mathbb R^N),	
\end{gathered}
\]
\[
\Delta^2\bar u-\Delta\bar u=\bar\lambda\bar u+\mu|\bar u|^{p-2}\bar u+|\bar u|^{4^*-2}\bar u,
\]
and
\[
P(\bar u)=0,\quad I(\bar u)=m_r(c),\quad 0<\|\bar u\|_2^2\leq c,\quad \bar\lambda\big(c-\|\bar u\|_2^2\big)=0
\]
If the sequence additionally satisfies condition \((N)\), then \(\bar\lambda<0\). Consequently, \(\|\bar u\|_2^2=c\) and \(u_n\to\bar u\) in \(H^2(\mathbb R^N)\).
\end{lemma}
\begin{proof}
For the sequence $\{u_n=\mathcal{H}(v_n,s_n)\}$ determined in Lemma \ref{Lem:PS-true},
making use of \cite[Proposition 5.12]{Willem1996}, we know that there exists $\{\lambda_n\}\subset \mathbb R$ such that
$$
I'(u_n)-\lambda_n\Psi'(u_n)\to 0\ \ \ \mbox{as}\ \ \ n\to\infty,
$$
where $\Psi$ is given in (\ref{eqn:Psi-define}). By the same dual-norm argument, there exists a sequence \(\varepsilon_n\to0\) such that, for every \(\varphi\in H_r^2(\mathbb R^N)\),
\begin{equation}\label{eqn:Jun-varphi-1}
\left|I'(u_n)[\varphi]-\lambda_n\int_{\mathbb R^N}u_n\varphi\,dx\right|
\leq\varepsilon_n\|\varphi\|_{H^2}.
\end{equation}
Due to the boundedness of $\{u_n\}$ from Lemma \ref{Lem:PS-true}-\emph{(iii)}, we see that
\begin{equation}\label{eqn:Jun-un-on1}
\|\Delta u_n\|^2_2+\|\nabla u_n\|^2_2-\mu\|u_n\|^p_p-\|u_n\|^{4^*}_{4^*}=\lambda_n\|u_n\|^2_2+o_n(1)=\lambda_n c+o_n(1),
\end{equation}
which indicates that $\{\lambda_n\}$ is also a bounded sequence. Meanwhile, by the boundedness of $\{u_n\} $ and the compact embedding $H^2_{r}({\mathbb R^N})\hookrightarrow L^p({\mathbb R^N})$, there exists $\bar{u}\in H^2_{r}({\mathbb R^N})$ such that, after passing to a subsequence,
\begin{equation}\label{eqn:un-uc-H2-Lp-RN}
u_n \rightharpoonup \bar{u}\ \ \mbox{in}\ \ H^2_{r}(\mathbb{R}^N),\ \ u_n\to \bar{u} \ \ \mbox{in}\ \ L^p(\mathbb{R}^N)\ \ \mbox{and}\ \ u_n\to \bar{u} \ \ \mbox{a.e. in}\ \ \mathbb{R}^N.
\end{equation}
Let $w_n:=u_n-\bar{u}$, then (\ref{eqn:un-uc-H2-Lp-RN}) states that
\begin{equation}\label{eqn:wn-H2-Lp-RN}
w_n \rightharpoonup 0\ \ \mbox{in}\ \ H^2_{r}(\mathbb{R}^N),\ \ w_n\to 0 \ \ \mbox{in}\ \ L^p(\mathbb{R}^N)\ \ \mbox{and}\ \ w_n\to 0 \ \ \mbox{a.e. in}\ \ \mathbb{R}^N.
\end{equation}
Since $\{\lambda_n\}$ is bounded, we are allowed to choose a subsequence such that $\lambda_n\to \bar{\lambda}\in\mathbb R$. 

We first show that the weak limit \(\bar u\) does not vanish identically. Suppose by contradiction that $\bar{u}=0$.
Since $\{u_n\}$ is bounded in $H^2_r(\mathbb R^N)$, up to a subsequence we can assume that $\|\Delta u_n\|^2_2+\frac{1}{2}\|\nabla u_n\|^2_2\to l\geq0$. Recalling that $\partial_s\tilde{I}(v_n,s_n)\to 0$
and $u_n\to 0$ in $L^p(\mathbb R^N)$, we have \(0\leq\|u_n\|^{4^*}_{4^*}=\|\Delta u_n\|^2_2+\frac{1}{2}\|\nabla u_n\|^2_2-\mu\gamma_p\|u_n\|^p_p+o_n(1)\to l\).
Therefore, in light of (\ref{eqn:Defn-S}), it gives that \(S\leq \frac{\|\Delta u_n\|^2_2}{\|u_n\|^2_{4^*}}\leq \frac{\|\Delta u_n\|^2_2+\frac{1}{2}\|\nabla u_n\|^2_2}{\|u_n\|^2_{4^*}}\),
which implies that either $l\geq S^\frac{N}{4}$ or $l=0$. If $l\geq S^\frac{N}{4}$, we infer that
$$
\begin{aligned}
m_r(c)+o_n(1)=I(u_n)
&\geq\frac{2}{N}\Big(\|\Delta u_n\|^2_2+\frac{1}{2}\|\nabla u_n\|^2_2\Big)-\mu\Big(\frac{1}{p}-\frac{\gamma_p}{4^*}\Big)\|u_n\|^p_p+o_n(1),
\end{aligned}
$$
since $I(u_n)\to m_r(c)$ and Lemma \ref{Lem:PS-true} \emph{(ii)}. Hence, one has \(m_r(c)\geq \frac{2}{N}l\geq\frac{2}{N}S^\frac{N}{4}\),
which contradicts with the assumption $m_r(c)<\frac{2}{N}S^\frac{N}{4}$. When $l=0$, we see that
$$
\|u_n\|_p\to 0,\ \ \|\Delta u_n\|_2\to 0,\ \ \|\nabla u_n\|_2\to 0\ \ \mbox{and}\ \ \  \|u_n\|^{4^*}_{4^*}\to 0\ \ \ \mbox{as}\ \ \ n\to\infty,
$$
which brings that $I(u_n)\to 0=m_r(c)$, an obvious contradiction. Hence, the weak limit $\bar{u}$ does not vanish identically. Passing to the limit in (\ref{eqn:Jun-varphi-1}), we see that $\bar{u}$ is a weak solution of
\begin{equation}\label{eqn:equation-1}
\Delta^2u-\Delta u=\bar{\lambda} u+\mu|u|^{p-2}u+|u|^{4^*-2}u\ \ \mbox{in}\ \ H_r^{-2}(\mathbb R^N).
\end{equation}
Testing (\ref{eqn:equation-1}) with \(\bar u\), we obtain
\[
\|\Delta\bar u\|_2^2+\|\nabla\bar u\|_2^2=\bar\lambda\|\bar u\|_2^2+\mu\|\bar u\|_p^p+\|\bar u\|_{4^*}^{4^*}.
\]
In addition, \cite[Lemma 2.1]{Bonheure-Casteras-Gou-Jeanjean2017} gives
\[
\begin{aligned}
\frac{N-4}{2}\|\Delta\bar u\|_2^2+\frac{N-2}{2}\|\nabla\bar u\|_2^2
=\frac{N\bar\lambda}{2}\|\bar u\|_2^2+\frac{N\mu}{p}\|\bar u\|_p^p
+\frac{N}{4^*}\|\bar u\|_{4^*}^{4^*}.
\end{aligned}
\]
Combining the two identities yields
\begin{equation}\label{eqn:Pohozaev-identify-baru}
	\begin{aligned}
		P(\bar{u})&=\|\Delta \bar{u}\|^2_2+\frac{1}{2}\|\nabla \bar{u}\|^2_2-\mu \gamma_p\|\bar{u}\|^p_p-\|\bar{u}\|^{4^*}_{4^*}=0
	\end{aligned}
\end{equation}
and
\begin{equation}\label{eqn:Pohozaev-identify-baru-I-ge0}
\begin{aligned}
I(\bar u)=\frac14\|\nabla\bar u\|_2^2+\mu\left(\frac{\gamma_p}{2}-\frac1p\right)\|\bar u\|_p^p+\frac2N\|\bar u\|_{4^*}^{4^*}>0.
\end{aligned}
\end{equation}
By virtue of (\ref{eqn:wn-H2-Lp-RN}), the Br\'{e}zis-Lieb Lemma (\!\!\cite[Lemma 1.32]{Willem1996}) ensures that
\begin{equation}\label{eqn:unuo(1)}
\|u_n\|^2_2=\|w_n\|^2_2+\| \bar{u}\|^2_2+o_n(1),
\end{equation}
\begin{equation}\label{eqn:unpupo(1)}
\|u_n\|^{p}_{p}=\|w_n\|^p_p+\|\bar{u}\|^p_p+o_n(1)
\end{equation}
and
\begin{equation}\label{eqn:un4-wnp-uc4}
\|u_n\|^{4^*}_{4^*}=\|w_n\|^{4^*}_{4^*}+\|\bar{u}\|^{4^*}_{4^*}+o_n(1).
\end{equation}
In addition, recall that $w_n=u_n-\bar{u}\rightharpoonup 0$ in $H^2_{r}(\mathbb{R}^N)$ and together with \cite[Page 18: lines 18-19]{Bonheure-Casteras-Gou-Jeanjean2019}, it also indicates that
\begin{equation}\label{eqn:DunDuo-nabla}
\|\nabla u_n\|^2_2=\|\nabla w_n\|^2_2+\|\nabla \bar{u}\|^2_2+o_n(1)
\end{equation}
and
\begin{equation}\label{eqn:DunDuo(1)}
\|\Delta u_n\|^2_2=\|\Delta w_n\|^2_2+\|\Delta \bar{u}\|^2_2+o_n(1).
\end{equation}
Therefore, based on Lemma \ref{Lem:PS-true} \emph{(ii)} and $u_n\to \bar{u}$ in $L^p(\mathbb R^N)$, we deduce from (\ref{eqn:un4-wnp-uc4})-(\ref{eqn:DunDuo(1)}) that
$$
\begin{aligned}
\|\Delta \bar{u}\|^2_2+\|\Delta w_n\|^2_2+\frac{1}{2}\|\nabla \bar{u}\|^2_2+\frac{1}{2}\|\nabla w_n\|^2_2
&=\mu \gamma_p\|\bar{u}\|^p_p+\|w_n\|^{4^*}_{4^*}+\|\bar{u}\|^{4^*}_{4^*}+o_n(1).
\end{aligned}
$$
Furthermore, according to (\ref{eqn:Pohozaev-identify-baru}), there holds that
$$
\|\Delta w_n\|^2_2+\frac{1}{2}\|\nabla w_n\|^2_2=\|w_n\|^{4^*}_{4^*}+o_n(1).
$$
So, up to a subsequence, it can be assumed that
\begin{equation}\label{eqn:Deltawn-wn4*-l}
\lim\limits_{n\to\infty}\Big(\|\Delta w_n\|^2_2+\frac{1}{2}\|\nabla w_n\|^2_2\Big)=\lim\limits_{n\to\infty}\|w_n\|^{4^*}_{4^*}=l\geq 0.
\end{equation}
Using (\ref{eqn:Defn-S}) again, we have \(S\leq \frac{\|\Delta w_n\|^2_2}{\|w_n\|^2_{4^*}}\leq \frac{\|\Delta w_n\|^2_2+\frac{1}{2}\|\nabla w_n\|^2_2}{\|w_n\|^2_{4^*}}\),
which means $l\geq S l^\frac{2}{4^*}$, that is, either $l=0$ or $l\geq S^\frac{N}{4}$. If $l\geq S^\frac{N}{4}$, (\ref{eqn:unpupo(1)})-(\ref{eqn:Deltawn-wn4*-l}) guarantee that
$$
\begin{aligned}
m_r(c)=\lim\limits_{n\to\infty}I(u_n)
&\geq I(\bar{u})+\Big(\frac{1}{2}-\frac{1}{4^*}\Big)l=I(\bar{u})+\frac{2}{N}l\geq I(\bar{u})+\frac{2}{N}S^\frac{N}{4}.
\end{aligned}
$$
Since \(I(\bar u)>0\), the above inequality contradicts the assumption \(m_r(c)<\frac{2}{N} S^\frac{N}{4}\). Therefore, \(l=0\) and (\ref{eqn:Deltawn-wn4*-l}) gives
\begin{equation}\label{eqn:wn-un-2-4*}
\lim\limits_{n\to\infty}\Big(\|\Delta w_n\|^2_2+\frac{1}{2}\|\nabla w_n\|^2_2\Big)=\lim\limits_{n\to\infty}\|w_n\|^{4^*}_{4^*}=0
\end{equation}
and \(I(\bar u)=m_r(c)\). To identify the possible \(L^2\) mass defect, test (\ref{eqn:Jun-varphi-1}) and (\ref{eqn:equation-1}) with  $w_n=u_n-\bar{u}$. Subtracting the resulting identities yields
\begin{equation}\label{eqn:un-uc-lambda=Lp-4}
\begin{aligned}
\int_{\mathbb R^N}&|\Delta(u_n-\bar{u})|^2+|\nabla(u_n-\bar{u})|^2-(\lambda_n u_n-\bar{\lambda} \bar{u})(u_n-\bar{u})dx\\
&=\mu\int_{\mathbb R^N}(|u_n|^{p-2}u_n-|\bar{u}|^{p-2}\bar{u})(u_n-\bar{u})dx+\int_{\mathbb R^N}(|u_n|^{4^*-2}u_n-|\bar{u}|^{4^*-2}\bar{u})(u_n-\bar{u})dx+o_n(1).
\end{aligned}
\end{equation}
Observing that $u_n\to \bar{u}$ in $L^p(\mathbb R^N)$ and $u_n\to \bar{u}$ in $L^{4^*}(\mathbb R^N)$ from (\ref{eqn:un-uc-H2-Lp-RN}) and (\ref{eqn:wn-un-2-4*}) respectively, and using the H\"{o}lder's inequality and \cite[Lemma 1.20]{Zou-Schechter2006}, we derive that
\begin{equation}\label{eqn:un-p2-un-uc-0}
\left|\int_{\mathbb R^N}(|u_n|^{p-2}u_n-|\bar{u}|^{p-2}\bar{u})(u_n-\bar{u})dx\right|\leq\||u_n|^{p-2}u_n-|\bar{u}|^{p-2}\bar{u}\|_\frac{p}{p-1}\|u_n-\bar{u}\|_p
\to 0
\end{equation}
and
\begin{equation}\label{eqn:I-alpha-to-0}
\left|\int_{\mathbb R^N}(|u_n|^{4^*-2}u_n-|\bar{u}|^{4^*-2}\bar{u})(u_n-\bar{u})dx\right|\leq\||u_n|^{4^*-2}u_n-|\bar{u}|^{4^*-2}\bar{u}\|_\frac{4^*}{4^*-1}\|u_n-\bar{u}\|_{4^*}
\to 0.
\end{equation}
Therefore, from (\ref{eqn:un-uc-lambda=Lp-4})-(\ref{eqn:I-alpha-to-0}), we obtain \(0=\bar\lambda\lim_{n\to\infty}\|u_n-\bar u\|_2^2\). Since \(u_n\rightharpoonup\bar u\) in \(L^2(\mathbb R^N)\) and \(\|u_n\|_2^2=c\), we have \(\lim_{n\to\infty}\|u_n-\bar u\|_2^2=c-\|\bar u\|_2^2\). Consequently, \(\bar\lambda\big(c-\|\bar u\|_2^2\big)=0\).
Finally, by (\ref{eqn:Jun-un-on1}) and \(P(u_n)\to0\), we have \(\lambda_n c=\frac12\|\nabla u_n\|_2^2-\mu(1-\gamma_p)\|u_n\|_p^p+o(1)\). If condition \((N)\) holds, passing to the limit gives \(\bar\lambda<0\). It follows that \(\|\bar u\|_2^2=c\). Hence, the weak convergence in \(L^2(\mathbb R^N)\) is strong. Combining this with (\ref{eqn:wn-un-2-4*}), we conclude that \(u_n\to\bar u\) in \(H^2(\mathbb R^N)\). This completes the proof.
\end{proof}
We next estimate the mountain pass level using a fixed radial test function. This argument does not require dimension-dependent bubble expansions.

\begin{lemma}\label{Lem:gamma-mu-estimate}
Let \(N\geq5\), \(c>0\) and \(\bar p<p<4^*\). Then, \(\lim_{\mu\to+\infty}m_r(c)=0\). In particular, there exists \(\mu_0=\mu_0(N,p,c)>0\) such that \(0<m_r(c)<\frac{2}{N} S^\frac{N}{4}\) for every \(\mu>\mu_0\).
\end{lemma}

\begin{proof}
Fix a function \(w\in C_c^\infty(\mathbb R^N)\cap S_r(c)\), independently of \(\mu\). Let \(s_\mu\) be the unique maximum point of the function \(s\mapsto I(H(w,s))\), whose existence follows from Lemma \ref{Prop:I-u-maximum}. For sufficiently large \(\mu\), we have \(P(w)<0\), which implies \(s_\mu<0\). Since \(P(H(w,s_\mu))=0\), it follows that
\[
\begin{aligned}
\mu\gamma_p\|w\|_p^pe^{2(p\gamma_p-1)s_\mu}
=e^{2s_\mu}\|\Delta w\|_2^2+\frac12\|\nabla w\|_2^2-e^{2(4^*-1)s_\mu}\|w\|_{4^*}^{4^*}\leq\|\Delta w\|_2^2+\frac12\|\nabla w\|_2^2.
\end{aligned}
\]
Consequently,
\[
e^{2s_\mu}
\leq
\left(
\frac{\|\Delta w\|_2^2+\frac12\|\nabla w\|_2^2}
{\mu\gamma_p\|w\|_p^p}
\right)^{\frac1{p\gamma_p-1}}
\rightarrow0.
\]
By Lemmas \ref{Prop:I-u-maximum} and \ref{Lem:gamma-tildegamma-equal}, we have
\(0<m_r(c)=\gamma_\mu(c)\leq I(H(w,s_\mu))\). Since \(s_\mu<0\), we further obtain that
\[
\begin{aligned}
m_r(c)
\leq\frac12e^{4s_\mu}\|\Delta w\|_2^2+\frac12e^{2s_\mu}\|\nabla w\|_2^2
\leq\frac12e^{2s_\mu}\left(\|\Delta w\|_2^2+\|\nabla w\|_2^2\right)\rightarrow0.
\end{aligned}
\]
This proves the desired limit. The strict upper bound follows immediately.
\end{proof}

\textbf{Proof of Theorem \ref{Thm:normalized-bScritical-solutions-L2super}.} Fix \(\mu>\mu_0\), where \(\mu_0\) is given by Lemma \ref{Lem:gamma-mu-estimate}. By Lemma \ref{Lem:PS-true}, there exists a bounded radial constrained Palais-Smale sequence \(\{u_n\}\subset S_r(c)\) at the level \(m_r(c)\), with \(P(u_n)\to0\). Moreover,
\(0<m_r(c)<\frac{2}{N} S^\frac{N}{4}\). Therefore, Lemma \ref{pro:solution-either-or-true} yields a nonzero radial weak solution \(u_{\mu,c}\) satisfying
\[
P(u_{\mu,c})=0,\qquad
I(u_{\mu,c})=m_r(c),
\]
and
\[
0<\|u_{\mu,c}\|_2^2\leq c,\qquad
\lambda_{\mu,c}\big(c-\|u_{\mu,c}\|_2^2\big)=0.
\]
If condition \((N)\) holds for the chosen Palais-Smale sequence, Lemma \ref{pro:solution-either-or-true} further gives \(\lambda_{\mu,c}<0\). Consequently, \(\|u_{\mu,c}\|_2^2=c\) and \(u_n\to u_{\mu,c}\) in \(H^2(\mathbb R^N)\). Thus \(u_{\mu,c}\) is a radial normalized weak solution of (\ref{eqn:BS-equation-L2-Super+Critical}), satisfying \(0<I(u_{\mu,c})<\frac{2}{N} S^\frac{N}{4}\). This completes the proof.
\\
\qed

\noindent\textbf{Disclosure of computational assistance.}
OpenAI's ChatGPT-6 was used as an auxiliary tool for language editing, algebraic and proof cross-checks, manuscript revision, and bibliographic checks. All mathematical statements, arguments, and references remain the responsibility of the authors.
\\[-0.5cm]

\noindent\textbf{Acknowledgment.}
The authors would like to thank Hang lyu and Guan-Qun Wang of Lanzhou University for their assistance with this paper.
\\[-0.5cm]

\noindent\textbf{Funding.}
H.R. Sun was partially supported by the NSFC (Grant No. 11671181) and NSF of Gansu Province of China (Grant No. 21JR7RA535, 24JRRA414). Z. Zhang was partially supported by the NSFC (Grant No. 12371402).
\\[-0.5cm]

\noindent\textbf{Authors' contributions.}
All authors have contributed equally.
\\[-0.5cm]

\noindent \textbf{Data Availability.}
The authors declare that data sharing is not applicable to this article as no data sets were generated or analyzed during the current study.
\\[-0.5cm]

\noindent \textbf{Conflict of interest.}
On behalf of all authors, the corresponding author states that there is no conflict of interest.

\bibliographystyle{elsarticle-num}

\end{document}